\documentclass[12pt,a4paper]{amsart}
\usepackage[czech,english]{babel}
\usepackage{amssymb,eucal}
\usepackage[cmtip,arrow]{xy}
\usepackage{pb-diagram,pb-xy} 
\usepackage{hyperref}

\calclayout

\newcommand{\+}{\nobreakdash-}
\newcommand{\<}{\nobreakdash--}

\renewcommand{\:}{\colon}
\renewcommand{\.}{\mskip.5\thinmuskip}

\newcommand{\rarrow}{\longrightarrow}

\newcommand{\ot}{\otimes}
\newcommand{\ocn}{\odot}
\newcommand{\tim}{\rightthreetimes} 
\newcommand{\ov}{\overline}

\DeclareMathOperator{\Hom}{Hom}

\DeclareMathOperator{\Ext}{Ext}
\DeclareMathOperator{\PExt}{PExt}
\DeclareMathOperator{\Ctrtor}{Ctrtor}
\DeclareMathOperator{\id}{id}
\DeclareMathOperator{\im}{im}

\newcommand{\lrarrow}{\.\relbar\joinrel\relbar\joinrel\rightarrow\.}

\newcommand{\modl}{{\operatorname{\mathsf{--mod}}}}

\newcommand{\contra}{{\operatorname{\mathsf{--contra}}}}
\newcommand{\discr}{{\operatorname{\mathsf{discr--}}}}

\newcommand{\Ab}{\mathsf{Ab}}
\newcommand{\Sets}{\mathsf{Sets}}
\newcommand{\Add}{\mathsf{Add}}
\newcommand{\proj}{{\mathsf{proj}}}

\newcommand{\R}{\mathfrak R}

\newcommand{\C}{\mathfrak C}
\newcommand{\D}{\mathfrak D}
\newcommand{\F}{\mathfrak F}
\newcommand{\G}{\mathfrak G}
\renewcommand{\P}{\mathfrak P}
\newcommand{\Q}{\mathfrak Q}
\renewcommand{\H}{\mathfrak H}
\newcommand{\B}{\mathfrak B}
\newcommand{\I}{\mathfrak I}
\newcommand{\J}{\mathfrak J}
\newcommand{\K}{\mathfrak K}
\renewcommand{\L}{\mathfrak L}
\newcommand{\fM}{\mathfrak M}
\newcommand{\fN}{\mathfrak N}

\newcommand{\ft}{\mathfrak t}
\newcommand{\fp}{\mathfrak p}
\newcommand{\ii}{\mathfrak i}
\newcommand{\fb}{\mathfrak b}
\newcommand{\fc}{\mathfrak c}
\newcommand{\fs}{\mathfrak s}

\newcommand{\fk}{\mathfrak k}
\newcommand{\fl}{\mathfrak l}
\newcommand{\fm}{\mathfrak m}
\newcommand{\fn}{\mathfrak n}
\newcommand{\fq}{\mathfrak q}
\newcommand{\fj}{\mathfrak j}
\newcommand{\fg}{\mathfrak g}
\newcommand{\fh}{\mathfrak h}
\newcommand{\ff}{\mathfrak f}

\newcommand{\blambda}{\boldsymbol\lambda}
\newcommand{\bgamma}{\boldsymbol\gamma}
\newcommand{\bbeta}{\boldsymbol\beta} 
\newcommand{\brho}{\boldsymbol\rho}

\newcommand{\Mat}{\mathfrak{Mat}}

\newcommand{\N}{\mathcal N}
\newcommand{\cS}{\mathcal S}
\newcommand{\cV}{\mathcal V}

\newcommand{\sA}{\mathsf A}
\newcommand{\sB}{\mathsf B}
\newcommand{\sC}{\mathsf C}
\newcommand{\sE}{\mathsf E}
\newcommand{\sL}{\mathsf L}

\newcommand{\Z}{{\mathbb Z}}

\theoremstyle{plain}
\newtheorem{thm}{Theorem}[section]
\newtheorem{prop}[thm]{Proposition}
\newtheorem{lem}[thm]{Lemma}

\newtheorem{conj}[thm]{Conjecture}
\newtheorem{cor}[thm]{Corollary}
\newtheorem{appl}[thm]{Application}
\theoremstyle{definition}

\newcommand{\Section}[1]{\bigskip\section{#1}\medskip}
\begin{document}

\title{Projective covers of flat contramodules}

\author[S.~Bazzoni]{Silvana Bazzoni}

\address[Silvana Bazzoni]{%
Dipartimento di Matematica ``Tullio Levi-Civita'' \\
Universit\`a di Padova \\
Via Trieste 63, 35121 Padova, Italy}

\email{bazzoni@math.unipd.it}

\author[L.~Positselski]{Leonid Positselski}

\address[Leonid Positselski]{%
Institute of Mathematics of the Czech Academy of Sciences \\
\v Zitn\'a~25, 115~67 Prague~1 \\ Czech Republic; and
\newline\indent Laboratory of Algebra and Number Theory \\
Institute for Information Transmission Problems \\
Moscow 127051 \\ Russia} 

\email{positselski@math.cas.cz}

\author[J.~\v S\v tov\'\i\v cek]{Jan \v S\v tov\'\i\v cek}

\address[Jan {\v S}{\v{t}}ov{\'{\i}}{\v{c}}ek]{%
Charles University in Prague, Faculty of Mathematics and Physics,
Department of Algebra, Sokolovsk\'a 83, 186 75 Praha,
Czech Republic}

\email{stovicek@karlin.mff.cuni.cz}

\begin{abstract}
 We show that a direct limit of projective contramodules (over a right
linear topological ring) is projective if it has a projective cover.
 A similar result is obtained for $\infty$\+strictly flat contramodules
of projective dimension not exceeding~$1$, using an argument based on
the notion of the topological Jacobson radical.
 Covers and precovers of direct limits of more general classes of
objects, both in abelian categories with exact and with nonexact
direct limits, are also discussed, with an eye towards the Enochs
conjecture about covers and direct limits, using locally split
(mono)morphisms as the main technique.
 In particular, we offer a simple elementary proof of
the Enochs conjecture for the left class of an $n$\+tilting cotorsion
pair in an abelian category with exact direct limits.
\end{abstract}

\maketitle

\tableofcontents

\section{Introduction}
\medskip

 The notion of a projective cover is dual to that of an injective
envelope.
 While injective envelopes exist in all Grothendieck abelian categories,
projective covers are more rare.
 It was shown in the classical paper of Bass~\cite{Bas} that all
left modules over an associative ring $R$ have projective covers if
and only if all flat left $R$\+modules are projective.
 Such rings were called \emph{left perfect} in~\cite{Bas}.
 Subsequently people realized that if a flat module over an associative
ring has a projective cover, then such module is projective (see,
e.~g., \cite[Section~36.3]{Wis}).

 The classical \emph{Govorov--Lazard theorem}~\cite{Gov,Laz} tells that
the flat modules are precisely the direct limits of (finitely generated)
projective modules.
 It is not known whether an analogue of this result holds for
contramodules.
 It is only clear that all the direct limits of projective contramodules
are flat.
 In fact, in Corollary~\ref{direct-limits-1-strictly-flat} we show that
all the direct limits of projective contramodules belong to a possibly
more narrow class of \emph{$1$\+strictly flat} contramodules.

 A \emph{Bass flat module} over an associative ring $R$ is a countable
direct limit of copies of the free $R$\+module with one generator
$R=R[*]$.
 All Bass flat modules have projective dimension at most~$1$.
 This class of modules and its generalizations played an important role
both in Bass' paper~\cite{Bas} and in subsequent works
(see, e.~g., the recent papers~\cite{Sar,AST}).
 In this paper, we consider the analogous class of contramodules over
a topological ring.
 Our results imply that a Bass flat contramodule cannot have
a projective cover unless it is projective.

 In fact, we prove two different generalizations of the latter claim,
provable by very different techniques.
 On the one hand, Theorem~\ref{1-strictly-flat-projdim1-cover-thm}
tells that an $\infty$\+strictly flat contramodule of projective
dimension not exceeding~$1$ is projective if it has a projective cover.
 The proof is based on the concept of the topological Jacobson radical
of a topological ring.
 On the other hand, by
Corollary~\ref{direct-limits-proj-contramodules-no-proj-covers},
the same assertion applies to an arbitrary (not necessarily countable)
direct limit of projective contramodules.
 The proof is based on considerations of local splitness.

 Speaking of the latter, one first of all has to distinguish between
\emph{locally split monomorphisms} and \emph{locally split epimorphisms}
of modules.
 These are two different theories (even if related by a vague analogy).

 The notion of a locally split monomorphism of modules seems to go
back to Chase's exposition of a result of
Villamayor~\cite[Proposition~2.2]{Ch}.
 Subsequently the property was studied by Rangaswamy with
coauthors~\cite{RR,JR}, who called such monomorphisms (or submodules)
``strongly pure''.
 The terminology ``locally split submodule'' appeared in
Azumaya's paper~\cite{Az1}.
 Zimmermann~\cite{Zim} uses the ``strongly pure monomomorphism''
terminology.
 It is the locally split \emph{monomorphisms} of modules that are
relevant for the categorical generalization developed in the present
paper.

 Locally split epimorphisms of modules were studied (under various names)
by Rangaswamy~\cite{Rang} and Azumaya~\cite{Az0,Az2}.
 An attempt to compare and formulate the connections between
the concepts of locally split monomorphisms and locally split
epimorphisms is made in the recent paper~\cite{YY}.

 The \emph{Enochs conjecture} (or ``a question of Enochs'') suggests
that any covering class of modules is closed under direct
limits~\cite[Section~5.4]{GT} (cf.~\cite[Section~5]{AST}).
 This problem was addressed in the papers~\cite{Sar,AST}, where some
results in the direction of a positive answer to the question of Enochs
were obtained.

 The following observation related to precovers and covers of direct
limits plays a key role in the present paper.
 Let $\sC$ be a class of modules (over a fixed associative ring $A$)
closed under direct sums and direct summands, and let $(C_x)_{x\in X}$
be a direct system of modules $C_x\in\sC$, indexed by a directed
poset~$X$.
 Consider the canonical presentation
$$
 0\lrarrow K\overset i\lrarrow\bigoplus_{x\in X}C_x\overset p\lrarrow
 \varinjlim_{x\in X}C_x\lrarrow0
$$
of the direct limit $D=\varinjlim_{x\in X}C_x$.
 Then the monomorphism~$i$ is locally split, that is, for every
element $k\in K$ there exists an $A$\+module morphism $g\:C'=
\bigoplus_{x\in X}C_x\rarrow K$ such that $gi(k)=k$.
 It follows that if the epimorphism $p$~is a $\sC$\+precover of $D$
and a morphism $q\:Q\rarrow D$ is a $\sC$\+cover of $D$, then
$q$~is an isomorphism and $p$~is a split epimorphism.
 We extend this observation first to abelian categories with exact
direct limits (in Theorem~\ref{ab5-direct-limit-precover-cover-thm}),
and subsequently, in some form, to cocomplete abelian categories with
nonexact direct limits
(Theorem~\ref{ab3-direct-limit-precover-cover-thm}).
 In greater generality, we discuss quasi-split exact sequences and
locally split (mono)morphisms in cocomplete abelian categories
in connection with covers.

 It was shown in~\cite[Lemma~2.1]{GG} that, whenever in the notation
above $X$ is a linearly ordered set, the kernel of the morphism~$p$ is
the union of a chain of direct summands in $C'=\bigoplus_{x\in X}C_x$.
 Our Proposition~\ref{direct-limit-epi-quasi-split} extends this result
to all directed posets $X$ (with the difference that the kernel of~$p$
is described as the union of a directed poset of direct summands
in~$C'$).
 Moreover, Proposition~\ref{direct-limit-epi-quasi-split} applies to
arbitrary cocomplete abelian categories.

 We also deduce the following application of topological algebra and
contramodule theory to the Enochs conjecture.
 Suppose that a left $A$\+module $M$ is what we call \emph{weakly
countably generated}; e.~g., this holds if $M$ is the sum of
a countable set of its \emph{dually slender} submodules (in
the sense of~\cite{Zem}).
 As usually, we denote by $\Add(M)$ the class of all direct summands
of direct sums of copies of~$M$.
 Assume further that for any countable direct system $M\rarrow M
\rarrow M\rarrow\dotsb$ of endomorphisms of $M$, the canonical
epimorphism $\bigoplus_{n=1}^\infty M\rarrow\varinjlim_{n\ge1} M=D$
is an $\Add(M)$\+precover of $D$, and that the $A$\+module $D$ has
an $\Add(M)$\+cover.
 Then the class of modules $\Add(M)$ is closed under direct limits.
 Moreover, the $A$\+module $M$ has a perfect decomposition
(in the sense of~\cite{AS}) in this case.
 This is the result of our
Application~\ref{weakly-countably-generated-application}.

 As another application to the Enochs conjecture, we demonstrate
a simple proof of the following assertion.
 Let $(\sL,\sE)$ be a cotorsion pair in an abelian category $\sA$
with exact direct limit functors.
 Suppose that the class $\sE$ is closed under direct limits
in~$\sA$.
 Let $M$ be an object in the kernel $\sL\cap\sE$ of the cotorsion pair.
 Assume that any direct limit of objects from $\Add(M)$ has
an $\Add(M)$\+cover in~$\sA$.
 Then the class of objects $\Add(M)$ is closed under direct limits
in~$\sA$.
 In the context of an $n$\+tilting cotorsion pair $(\sL,\sE)$, it
follows that the class of objects $\sL\subset\sA$ is closed under
direct limits whenever it is covering.
 Thus we recover some of the results of the paper~\cite{AST} with
our elementary methods (see Application~\ref{cotorsion-pair-appl}
and Corollary~\ref{tilting-corollary}).

\Section{Preliminaries}
 Throughout the paper, by ``direct limits'' we mean directed colimits,
i.~e., colimits indexed over directed posets.

 We refer to~\cite[Section~6]{PS1} or~\cite[Sections~1\<2]{Pproperf}
(see also~\cite[Introduction and Sections~5\<6]{PR}
or~\cite[Section~2]{Pcoun}) for the background material.
 This section only contains a sketch of the basic definitions and
a little discussion.

 A topological ring is said to be \emph{right linear} if its open right
ideals form a base of neighborhoods of zero.
 A right linear topological ring $\R$ is said to be \emph{separated}
if the natural map $\R\rarrow\varprojlim_{\I\subset\R}\R/\I$, where
$\I$ ranges over the open right ideals in $\R$, is injective; and $\R$
is said to be \emph{complete} if this map is surjective.
 Throughout this paper, $\R$ denotes a complete, separated right
linear topological ring.

 For any abelian group $A$ and a set $X$, we use $A[X]=A^{(X)}$ as
a notation for the direct sum of $X$ copies of~$A$.
 Elements of the group $A[X]$ are interpreted as finite formal linear
combinations of elements of $X$ with the coefficients in~$A$.
 For any set $X$, we denote by $\R[[X]]=\varprojlim_{\I\subset\R}
(\R/\I)[X]\subset \R^X$ the set of all infinite formal linear
combinations $\sum_{x\in X}r_xx$ of elements of $X$ with
the families of coefficients $(r_x\in\R)_{x\in X}$ converging to zero
in the topology of~$\R$.
 The latter condition means that, for every open right ideal
$\I\subset\R$, the subset $\{x\in X\mid r_x\notin\I\}\subset X$
is finite \cite[Section~6]{PS1}, \cite[Sections~1.5\<1.7]{Pproperf},
\cite[Section~2.7]{Pcoun}, \cite[Section~5]{PR}.

 The assignment of the set $\R[[X]]$ to an arbitrary set $X$ is
a covariant endofunctor on the category of sets,
$\R[[{-}]]\:\Sets\rarrow\Sets$.
 For any map of sets $f\:X\rarrow Y$, the induced map of sets
$\R[[f]]\:\R[[X]]\rarrow\R[[Y]]$ assigns to a formal linear combination
$\sum_{x\in X}r_xx$ the formal linear combination $\sum_{y\in Y}
s_yy$ with the coefficients $s_y=\sum_{x\in X}^{f(x)=y}r_x$.
 Here the infinite sum in the right-hand side is understood as the limit
of finite partial sums in the topology of~$\R$.

 For any set $X$, there is a natural ``point measure'' map $\epsilon_X\:
X\rarrow\R[[X]]$, assigning to every element $x\in X$ the formal
linear combination $\sum_{y\in X}r_yy$ with $r_y=1$ for $y=x$ and
$r_y=0$ otherwise.
 Moreover, for any set $X$ there is a natural ``opening of parentheses''
map $\phi_X\:\R[[\R[[X]]]]\rarrow\R[[X]]$ assigning a formal linear
combination to a formal linear combination of formal linear
combinations.
 The map~$\phi_X$ computes the products of pairs of elements in $\R$
and then the infinite sums of such products, interpreted as the limits
of finite partial sums in the topology of~$\R$.

 The functor $\R[[{-}]]$ endowed with the natural transformations
$\phi$ and~$\epsilon$ is a \emph{monad} on the category of sets.
 A \emph{left\/ $\R$\+contramodule} is defined as an algebra (or, in
our preferred terminology, a module) over this monad.
 In other words, a left $\R$\+contramodule $\C$ is a set endowed
with a \emph{left $\R$\+contraaction} map $\pi_\C\:\R[[\C]]\rarrow\C$
satisfying the conventional associativity and unitality equations of
an algebra/module over a monad $(\R[[{-}]],\phi,\epsilon)$.
 Informally, one can say that a left $\R$\+contramodule is a left
$\R$\+module endowed with infinite summation operations with
the families of coefficients converging to zero in the topology of~$\R$.

 The category of left $\R$\+contramodules $\R\contra$ is a locally
presentable abelian category with enough projective objects.
 There is an exact, faithful forgetful functor $\R\contra\rarrow\R\modl$
from the category of left $\R$\+contramodules to the category of
left $\R$\+modules; this functor preserves infinite products (but not
coproducts).
 The free $\R$\+contramodule with one generator $\R[[*]]=\R$ is
a natural projective generator of $\R\contra$.
 More generally, the projective $\R$\+contramodules are precisely
the direct summands of the \emph{free\/ $\R$\+contramodules}
$\R[[X]]$, where $X$ is an arbitrary set and the contraaction map
is $\pi_{\R[[X]]}=\phi_X$.
 For any left $\R$\+contramodule $\C$ and any set $X$ the group of
$\R$\+contramodule morphisms $\R[[X]]\rarrow\C$ is naturally
isomorphic to the group of all maps of sets $X\rarrow\C$.

  A right $\R$\+module $\N$ is said to be \emph{discrete} if
the annihilator of any element in $\N$ is an open right ideal in~$\R$
\cite[Section~7.2]{PS1}, \cite[Section~1.4]{Pproperf},
\cite[Section~2.3]{Pcoun}.
 Discrete right $\R$\+modules form a Grothendieck abelian category,
which we denote by $\discr\R$.

 Given a discrete right $\R$\+module $\N$ and an abelian group $V$,
the left $\R$\+module $\C=\Hom_\Z(\N,V)$ has a natural left
$\R$\+contramodule structure with the contraaction map given by
the rule
$$
 \left\langle b,
 \pi_\C\left(\sum\nolimits_{c\in\C}r_cc\right)
 \right\rangle=
 \sum\nolimits_{c\in\C}\langle br_c,c\rangle
 \qquad\text{for all $b\in\N$ and
 $\sum\nolimits_{c\in\C} r_cc\in\R[[\C]]$},
$$
where $\langle\ ,\ \rangle\:\N\times\Hom_\Z(\N,V)\rarrow V$
denotes the evaluation pairing.
 The sum in the right-hand side is finite because the annihilator
of~$b$ is open in $\R$ and the family of elements $(r_c\in\R)_{c\in\C}$
converges to zero.

 For any discrete right $\R$\+module $\N$ and any left
$\R$\+contramodule $\P$, the \emph{contratensor product}
$\N\ocn_\R\P$ is an abelian group constructed as the cokernel of
(the difference of) the natural pair of maps
$$
 \N\ot_\Z\R[[\P]]\,\rightrightarrows\,\N\ot_\Z\P.
$$
 Here one of the two maps is $\id_\N\ot\pi_\P$, while the other one is
given by the formula
$$
 b\ot\sum_{p\in\P}r_pp\longmapsto \sum_{p\in\P}br_p\ot p
 \qquad\text{for all $b\in\N$ and $\sum_{p\in\P}r_pp\in\R[[\P]]$},
$$
where, once again, the sum in the right-hand side is finite because
the annihilator of~$b$ is open in $\R$ and the family of elements
$r_p\in\R$ converges to zero~\cite[Section~7.2]{PS1},
\cite[Section~1.8]{Pproperf}, \cite[Section~2.8]{Pcoun},
\cite[Section~5]{PR}.

 For any discrete right $\R$\+module $\N$, any left $\R$\+contramodule
$\P$, and an abelian group $V$ there is a natural isomorphism of
abelian groups
\begin{equation} \label{contratensor-adjunction}
 \Hom^\R(\P,\Hom_\Z(\N,V))\,\simeq\,\Hom_\Z(\N\ocn_\R\P,\>V),
\end{equation}
where $\Hom^\R$ denotes the group of morphisms in the category
$\R\contra$.
 For any discrete right $\R$\+module $\N$ and any set $X$ there is
a natural isomorphism of abelian groups
\begin{equation} \label{contratensor-with-free}
 \N\ocn_\R\R[[X]]\,\simeq\,\N[X].
\end{equation}

 A left $\R$\+contramodule $\F$ is said to be \emph{flat} if the functor
of contratensor product ${-}\ocn_\R\F\:\discr\R\rarrow\Ab$, acting
from the category of discrete right $\R$\+modules to the category of
abelian groups $\Ab$, is exact.
 It is clear from the natural isomorphism~\eqref{contratensor-with-free}
that free (hence also projective) left $\R$\+contramodules are flat.
 Furthermore, it follows from the adjunction
isomorphism~\eqref{contratensor-adjunction} that the functor of
contratensor product ${-}\ocn_\R\nobreak{-}\:\discr\R\times\R\contra
\rarrow\Ab$ preserves colimits (in both of its arguments).
 It follows that the class of all flat left $\R$\+contramodules is
closed under direct limits.
 Hence all the direct limits of projective contramodules are flat.

 For the purposes of the present paper, an (apparently) stronger flatness
property of contramodules is relevant.
 The left derived functor of contratensor product
$$
 \Ctrtor_*^\R:\discr\R\times\R\contra\lrarrow\Ab,
 \qquad\Ctrtor_0^\R(\N,\C)=\N\ocn_\R\C
$$
is constructed using projective resolutions of its second (contramodule)
argument.
 A left $\R$\+contramodule $\F$ is said to be \emph{$n$\+strictly flat}
(where $n\ge1$ is an integer) if $\Ctrtor_i^\R(\N,\F)=0$ for all discrete
right $\R$\+modules $\N$ and all $0<i\le n$ \cite[Section~2]{Pproperf}.
 It suffices to check these conditions for the cyclic discrete right
$\R$\+modules $\N=\R/\I$.

 An $\R$\+contramodule is \emph{$\infty$\+strictly flat} if it is
$n$\+strictly flat for all $n\ge1$.
 Obviously, a contramodule of projective dimension~$\le n$ is
$n$\+strictly flat if and only if it is $\infty$\+strictly flat.

 The kernel of an epimorphism from an $n$\+strictly flat contramodule
to an $(n+1)$-strictly flat contramodule is $n$\+strictly flat.
 The kernel of an epimorphism from a flat contramodule to
a $1$\+strictly flat contramodule is flat.
 Any $1$\+strictly flat contramodule is flat (so one can think of flat
contramodules as ``$0$\+strictly flat'').

 Clearly, for every $n\ge1$ the class of all $n$\+strictly flat left
$\R$\+contramodules is closed under extensions.
 By~\cite[Lemma~2.1]{Pproperf}, the class of all $1$\+strictly flat
contramodules is also closed under coproducts.
 We will see below in Corollary~\ref{direct-limits-1-strictly-flat}
that the class of $1$\+strictly flat contramodules is closed under
direct limits.
 So all the direct limits of projective contramodules are, in fact,
$1$\+strictly flat.

 Over a topological ring $\R$ with a countable base of neighborhoods
of zero, any flat contramodule is $\infty$\+strictly
flat~\cite[Remark~6.11 and Corollary~6.15]{PR}.

 We will use the following pieces of notation
from~\cite[Section~5]{PR} and~\cite[Sections~1.5 and~1.10]{Pproperf}.
 Given a closed right ideal $\J\subset\R$ and a set $X$, we denote by
$\J[[X]]\subset\R[[X]]$ the subgroup of all zero-convergent $X$\+indexed
families of elements of $\J$ in the group of all such families of
elements of~$\R$.
 For any left $\R$\+contramodule $\C$, we denote by $\J\tim\C\subset\C$
the image of the restriction $\J[[\C]]\rarrow\C$ of the contraaction map
$\R[[\C]]\rarrow\C$ to the subgroup $\J[[\C]]\subset\R[[\C]]$.

 When $\J$ is a closed two-sided ideal in $\R$, the subgroup $\J\tim\C$
is actually an $\R$\+subcontramodule in~$\C$.
 For any closed right ideal $\J\subset\R$ and any set $X$, one has
$$
 \J\tim(\R[[X]])=\J[[X]]\subset\R[[X]].
$$ 
 For any open right ideal $\I\subset\R$ and any left $\R$\+contramodule
$\C$, the abelian group $\C/(\I\tim\C)$ can be interpreted as
the contratensor product
$$
 \C/(\I\tim\C)=(\R/\I)\ocn_\R\C
$$
of the cyclic discrete right $\R$\+module $\R/\I$ with
the left $\R$\+contramodule~$\C$ \cite[Section~1.10]{Pproperf},
\cite[Section~2.8]{Pcoun}, \cite[Section~5]{PR}.

\Section{Jacobson Radical and Superfluous Subcontramodules}
\label{jacobson-secn}

 Let $\sA$ be a category and $\sL\subset\sA$ be a class of objects.
 A morphism $l\:L\rarrow A$ in $\sA$ is said to be
an \emph{$\sL$\+precover} (of the object~$A$) if $L\in\sL$ and for any
morphism $l'\:L'\rarrow A$ with $L'\in\sL$ there exists a morphism
$f\:L'\rarrow L$ such that $l'=lf$.
 An $\sL$\+precover $l\:L\rarrow A$ is said to be an \emph{$\sL$\+cover}
if for any endomorphism $f\:L\rarrow L$ the equation $lf=l$ implies
that $f$~is an automorphism.

 Let $\sB$ be an abelian category with enough projective objects.
 We denote the full subcategory of projective objects in $\sB$
by $\sB_\proj\subset\sB$.
 Then a morphism $p\:P\rarrow B$ in $\sB$ is a projective precover
(i.~e., a $\sB_\proj$\+precover) if and only if $P\in\sB_\proj$ and
$p$~is an epimorphism.
 A projective precover $p\:P\rarrow B$ is a projective cover if and
only if its kernel $K$ is a superfluous subobject in~$P$.
 Here a subobject $K\subset P$ of an arbitrary object $P$ in an abelian
category $\sB$ is said to be \emph{superfluous} if for any subobject
$X\subset P$ the equality $K+X=P$ implies $X=P$.

 The aim of this section is to prove the following

\begin{thm} \label{1-strictly-flat-projdim1-cover-thm}
 Let\/ $\F$ be an\/ $\infty$\+strictly flat left\/ $\R$\+contramodule
of projective dimension not exceeding\/~$1$.
 Assume that\/ $\F$ has a projective cover in\/ $\R\contra$.
 Then\/ $\F$ is a projective\/ $\R$\+contramodule.
\end{thm}

 Our proof extends to the contramodule realm the argument for a discrete
ring $R$ outlined in the now-obsolete preprint~\cite[Lemma~3.2 and/or
Corollary~3.4(a)]{BPobsolete}.
 The proof is based on three technical propositions, the first of which
is formulated immediately below.

 We denote by $\H=\H(\R)\subset\R$ the topological Jacobson radical of
the ring $\R$, that is, the intersection of all the open maximal right
ideals in~$\R$ \cite[Section~3.B]{IMR}, \cite[Section~6]{Pproperf}.
 So $\H$ is a closed two-sided ideal in $\R$ \cite[Lemma~6.1]{Pproperf}.
 The Jacobson radical of the ring $\R$ viewed as an abstract
(nontopological) associative ring is denoted by $H=H(\R)\subset\R$.
 So $H$ is a two-sided ideal in $\R$, but we do not know whether it
needs to be a closed ideal.
 Obviously, one has $H(\R)\subset\H(\R)$.

\begin{prop} \label{superfluous-subcontramodule}
 Let\/ $\P$ be a projective left\/ $\R$\+contramodule and\/
$\K\subset\P$ be a superfluous subcontramodule.
 Then\/ $\K\subset\H\tim\P$.
\end{prop}

 The proof of Proposition~\ref{superfluous-subcontramodule} consists
of three lemmas.

\begin{lem} \label{superfluous-image}
 Let\/ $\sB$ be an abelian category, $f\:P\rarrow Q$ be a morphism in\/
$\sB$, and $K\subset P$ be a superfluous subobject.
 Then the image $L=f(K)$ of the subobject $K$ under the morphism~$f$
is a superfluous subobject in~$Q$.
 In particular, \par
\textup{(a)} if $P$, $Q\in\sB$ are two objects and $K\subset P$ is
a superfluous subobject, then $K\oplus0$ is a superfluous subobject
in $F=P\oplus Q$, \par
\textup{(b)} if $P$, $Q\in\sB$ are two objects, $K\subset F=P\oplus Q$
is a superfluous subobject, and $f\:F\rarrow Q$ is the direct summand
projection, then $f(K)$ is a superfluous subobject in~$Q$.
\end{lem}

\begin{proof}
 Let $Y\subset Q$ be a subobject such that $L+Y=Q$.
 Then $X=f^{-1}(Y)\subset P$ is a subobject such that $K+X=P$.
 Hence $X=P$; so $f(P)\subset Y$.
 It follows that $L=f(K)\subset f(P)\subset Y$ and therefore $Y=Q$.
\end{proof}

\begin{lem} \label{cyclic-subcontramodule}
 Let\/ $\C$ be a left\/ $\R$\+contramodule and $c\in\C$ be
an element.
 Then the cyclic\/ $\R$\+submodule\/ $\R c\subset\C$ is
an $\R$\+subcontramodule in\/~$\C$.
\end{lem}

\begin{proof}
 Let $\R=\R[[*]]$ be the free left\/ $\R$\+contramodule with one
generator.
 Then the $\R$\+contramodule morphisms $\R\rarrow\C$ correspond
bijectively to the elements of~$\C$.
 In other words, the map $\R\overset c\rarrow\C$ taking every element
$r\in\R$ to the element $rc\in\C$ is a left $\R$\+contramodule
morphism (see~\cite[Section~6.2]{PS1} or~\cite[Section~1.7]{Pproperf}).
 Now the cyclic submodule $\R c\subset\C$ is the image of this
contramodule morphism, hence it is a subcontramodule.
\end{proof}

 More generally, any finitely generated $\R$\+submodule of
an $\R$\+contramodule is a subcontramodule.

\begin{lem} \label{superfluous-in-rank-1-free}
 Consider the free left\/ $\R$\+contramodule with one generator\/
$\R[[*]]=\R$.
 Let\/ $\L\subset\R$ be a superfluous left\/ $\R$\+subcontramodule.
 Then\/ $\L\subset H(\R)$.
\end{lem}

\begin{proof}
 Notice first of all that $\L$ is a left $\R$\+submodule (i.~e.,
a left ideal) in~$\R$.
 Suppose $\L$ is not contained in $H(\R)$.
 Then there exists a (not necessarily closed) maximal left ideal
$M\subset\R$ such that $\L$ is not contained in~$M$, and
consequently $\L+M=\R$.
 Let $l\in\L$ and $m\in M$ be a pair of elements such that $l+m=1$.
 Then we have $\L+\R m$=$\R$ and $\R m\varsubsetneq\R$.
 By Lemma~\ref{cyclic-subcontramodule}, the principal left ideal
$\R m$ is a left $\R$\+subcontramodule in~$\R$.
 (Notice that there is \emph{no} claim about $\R m$ being
a \emph{closed} left ideal in~$\R$ here.)
 The contradiction with the superfluousness assumption proves that
$\L\subset H(\R)$.
\end{proof}

\begin{proof}[Proof of Proposition~\ref{superfluous-subcontramodule}]
 By Lemma~\ref{superfluous-image}(a), we can assume that $\P$ is
a free left $\R$\+con\-tramodule, $\P=\R[[X]]$.
 For every element $x\in X$, consider the coordinate projection
$f_x\:\P\hookrightarrow\R^X\rarrow\R$ corresponding to~$x$ and put
$\L_x=f_x(\K)$.
 By Lemma~\ref{superfluous-image}(b), \,$\L_x$ is a superfluous
left $\R$\+subcontramodule in~$\R$.
 According to Lemma~\ref{superfluous-in-rank-1-free}, we have
$\L_x\subset H\subset\H$.
 It follows that $\K\subset\H[[X]]=\H\tim\R[[X]]$.
\end{proof}

 In fact, we have shown that under the assumptions of
Proposition~\ref{superfluous-subcontramodule} one has
$\K\subset\overline H\tim\P$, where $\overline H\subset\R$ is
the topological closure of the abstract Jacobson radical $H\subset\R$
(but we will not use this fact).

 The second main technical ingredient is the next
 
\begin{prop} \label{1-strictly-flat-projdim1-prop}
 Let $0\rarrow\K\rarrow\P\rarrow\F\rarrow0$ be a short exact sequence of
left\/ $\R$\+contramodules, where\/ $\K$ and\/ $\P$ are projective\/
$\R$\+contramodules and\/ $\F$ is a\/ $1$\+strictly flat\/
$\R$\+contramodule.
 Let\/ $\J\subset\R$ be a closed right ideal.
 Then one has
$$
 \K\cap(\J\tim\P)=\J\tim\K.
$$ 
\end{prop}

\begin{proof}
 Let us first consider the case of an open right ideal $\I\subset\R$.
 Then the equation $\K\cap(\I\tim\P)=\I\tim\K$ is equivalent to
exactness of the short sequence
\begin{equation} \label{reduction-sequence}
 0\lrarrow\K/(\I\tim\K)\lrarrow\P/(\I\tim\P)\lrarrow\F/(\I\tim\F)
 \lrarrow0.
\end{equation}
 For any left $\R$\+contramodule $\C$, we have $\C/(\I\tim\C)=
(\R/\I)\ocn_\R\C$.
 In view of the homological long exact sequence of the derived
functor $\Ctrtor$,
$$
 \dotsb\lrarrow\Ctrtor^\R_1(\N,\F)\lrarrow\N\ocn_\R\K
 \lrarrow\N\ocn_\R\P\lrarrow\N\ocn_\R\F\lrarrow0,
$$
which is defined for any discrete right $\R$\+module $\N$ and in
particular for $\N=\R/\I$, it follows that the short
sequence~\eqref{reduction-sequence} is exact.

 In the general case of a closed right ideal $\J\subset\R$, we have
$$
 \J=\bigcap\nolimits_{\J\subset\I}\I,
$$ 
where the intersection is taken over all the open right ideals $\I$
in $\R$ containing~$\J$.
 It follows that, for any set~$X$,
$$
 \J[[X]]=\bigcap\nolimits_{\J\subset\I}\I[[X]]\,\subset\,\R[[X]].
$$
 In other words, for any free left $\R$\+contramodule $\G=\R[[X]]$
we have
$$
 \J\tim\G=\J[[X]]\,=\,\bigcap\nolimits_{\J\subset\I}\I[[X]]=
 \bigcap\nolimits_{\J\subset\I}\I\tim\G.
$$
 Since any projective left $\R$\+contramodule $\Q$ is a direct summand
of a free one, we obtain the equality
$$
 \J\tim\Q=\bigcap\nolimits_{\J\subset\I}\I\tim\Q\,\subset\,\Q.
$$
 In particular, this holds for $\Q=\P$ and $\Q=\K$.
 Finally, we can compute
$$
 \K\cap(\J\tim\P)=\K\cap\bigcap\nolimits_{\J\subset\I}\I\tim\P=
 \bigcap\nolimits_{\J\subset\I}\K\cap(\I\tim\P)=
 \bigcap\nolimits_{\J\subset\I}\I\tim\K=\J\tim\K.
$$
\end{proof}

 Our third main technical result in this section is the following
version of Nakayama lemma for projective contramodules.
 It is the contramodule generalization of
the classical~\cite[Proposition~2.7]{Bas}.

\begin{prop} \label{projective-contra-Nakayama}
 Let\/ $\P$ be a nonzero projective left\/ $\R$\+contramodule.
 Then
$$
 \H\tim\P\varsubsetneq\P.
$$
\end{prop}

 The proof of Proposition~\ref{projective-contra-Nakayama} is based on
two lemmas.
 The first of them expands the list of equivalent conditions
characterizing the topological Jacobson radical $\H$ of a topological
ring $\R$ given in~\cite[Lemma~6.2]{Pproperf}.

\begin{lem} \label{Jacobson-additionally-characterized}
 Given an element $h\in\R$, the following three conditions are
equivalent: \par
\textup{(a)} $h\in\H$; \par
\textup{(b)} for every open right ideal\/ $\I\subset\R$ and every
element $r\in\R$, the right multiplication by\/ $1-hr$ acts injectively
in\/ $\R/\I$; \par
\textup{(c)} for every open right ideal\/ $\I\subset\R$ and every
element $r\in\R$, the right multiplication by\/ $1-rh$ acts injectively
in\/ $\R/\I$. \par
 In particular, for every element $h\in\H$, the right multiplication
with\/ $1-h$ acts injectively in\/~$\R$.
\end{lem}

\begin{proof}
 (a) $\Longrightarrow$ (b) and~(c): since $\H$ is a two-sided ideal
in $\R$ by~\cite[Lemma~6.1]{Pproperf}, it suffices to show that $1-h$
acts injectively in $\R/\I$ for every $h\in\H$ and any open
right ideal~$\I$.
 Let $0\ne s+\I\in\R/\I$.
 Then there exists an open right ideal $\J\subset\R$ such that
$s\J\subset\I$.
 By~\cite[Lemma~6.2(iii) or~(iv)]{Pproperf}, there is an element
$t\in\R$ such that $(1-h)t+\J=1+\J$.
 Multiplying the latter equation by~$s$ on the left, we get
$s(1-h)t+s\J=s+s\J$, hence $s(1-h)t+\I=s+\I\ne0$ in $\R/\I$.
 It follows that $s(1-h)+\I\ne0$ in $\R/\I$.
 
 (b) or (c) $\Longrightarrow$ (a): by~\cite[Lemma~6.2(ii)]{Pproperf},
for any $h\notin\H$ there exists a simple discrete right $\R$\+module
$\cS$ and a pair of nonzero elements $x$, $y\in\cS$ such that $xh=y$.
 Since $\cS$ is simple, there is also an element $r\in\R$ such that
$yr=x$.
 Thus $x(1-hr)=0=y(1-rh)$ in $\cS$, and we have shown that neither
$1-hr$ nor $1-rh$ act injectively in the cyclic discrete right
$\R$\+module~$\cS$.

 To prove the last assertion of the lemma, suppose that we have
$s(1-h)=0$ in $\R$ for some elements $s\in\R$ and $h\in\H$.
 Let $\I\subset\R$ be an open right ideal.
 Then we have $s(1-h)+\I=0$ in $\R/\I$, which implies $s+\I=0$ in
$\R/\I$ by (b) or~(c).
 Hence $s\in\I$.
 As this holds for every open right ideal $\I$ and the topological ring
$\R$ is separated by assumption, we can conclude that $s=0$ in~$\R$.
\end{proof}

 For the next lemma we need the construction of the topological ring of
\emph{row-zero-convergent matrices} $\Mat_Y(\R)$ (see~\cite[Section~5]{PS3}).
 For any set $Y$, the elements of $\Mat_Y(\R)$ are $Y\times Y$ matrices
$(m_{x,y}\in\R)_{x,y\in Y}$ such that for every fixed $x\in Y$
the family of elements $(m_{x,y})_{y\in Y}$ converges to zero in
the topology of~$\R$. The usual matrix multiplication, which is well-defined thanks to the infinite summation in $\R$, gives $\Mat_Y(\R)$ a ring structure.
There is also a natural topology in $\Mat_Y(\R)$, which makes it a complete, separated topological ring with a base of neighborhoods of zero consisting of open right ideals.

The main motivation for the construction of $\Mat_Y(\R)$ were Morita equivalence type results.
In particular, following~\cite[Proposition~5.2]{PS3}, the categories of discrete right
modules over the rings $\R$ and $\Mat_Y(\R)$ are naturally equivalent.
The equivalence is provided by the functor
\begin{equation} \label{discrete-Morita}
\cV_Y\:\discr\R\lrarrow\discr\Mat_Y(\R)
\end{equation}
assigning to every discrete right $\R$\+module $\N$ the discrete right
$\Mat_Y(\R)$\+module $\cV_Y(\N)=\N^{(Y)}$.
Elements of the direct sum $\N^{(Y)}$ of $Y$ copies of $\N$ are
interpreted as finite rows of elements of $\N$, and $\Mat_Y(\R)$ acts
in $\cV_Y(\N)$ by the usual right action of matrices on row-vectors
(which is well-defined in this case due to the row-zero-convergence
condition imposed on the elements of $\Mat_Y(\R)$ and the discreteness
condition imposed on~$\N$).

\begin{lem} \label{matrix-Jacobson}
 The topological Jacobson radical of the topological ring\/ $\Mat_Y(\R)$
consists of all the (row-zero-convergent) matrices with entries in\/
the topological Jacobson radical\/ $\H$ of the ring\/~$\R$.
 So\/ $\H(\Mat_Y(\R))=\Mat_Y(\H(\R))$.
\end{lem}

\begin{proof}
 By~\cite[Lemma~6.2(ii)]{Pproperf}, the topological Jacobson radical
$\H(\R)$ consists of all the elements $a\in\R$ annihilating all
the simple discrete right $\R$\+modules, and similarly for
$\H(\Mat_Y(\R))$.

 The equivalence $\cV_Y\:\discr\R\rarrow\discr\Mat_Y(\R)$ in~\eqref{discrete-Morita}, as any equivalence of abelian categories,
induces a bijection between the isomorphism classes of simple objects.
 So the simple discrete right $\Mat_Y(\R)$\+modules are precisely
the right $\Mat_Y(\R)$\+modules $\cV_Y(\cS)$, where $\cS$
ranges over all the simple discrete right $\R$\+modules.
 It remains to observe that a matrix $A=(a_{xy})_{x,y\in X}
\in\Mat_Y(\R)$ annihilates all the elements of $\cV_Y(\cS)=\cS^{(Y)}$
if and only if all the entries $a_{x,y}$ of $A$ annihilate all
the elements of~$\cS$.
\end{proof}

\begin{proof}[Proof of Proposition~\ref{projective-contra-Nakayama}]
 We follow the argument in~\cite[proof of Proposition~17.14]{AF} with
suitable modifications.
 Let $\G=\R[[X]]$ be a free left $\R$\+contramodule and $\P$ be
a direct summand of~$\G$.
 We will view $\P$ as a subcontramodule in $\G$ and denote by
$e\:\G\rarrow\G$ an idempotent $\R$\+contramodule endomorphism of $\G$
such that $\P=\G e$ (for simplicity of notation, we let~$e$ act in
$\G$ on the right).
 Elements of the set $X$ will be viewed as (the basis) elements of~$\G$.

 Let $q=\sum_{x\in X}q_xx$ be an element of~$\P$.
 Here $(q_x\in\R)_{x\in X}$ is a family of elements converging to zero
in the ring $\R$ and the sum $\sum_{x\in X}q_xx$ can be understood as
the result of applying the contramodule infinite summation operation
with the family of coefficients~$q_x$ to the $X$\+indexed family of
elements $x\in\G$.

 Assuming that $\P=\H\tim\P$, we will prove that $q=0$.
 Indeed, we have $\H\tim\P\subset\H\tim\G=\H[[X]]$; so $\P\subset\H[[X]]
\subset\R[[X]]=\G$.
 For every element $x\in X$, we have $xe\in\P$, hence
$$
 xe=\sum\nolimits_{y\in X}a_{x,y}y, \qquad
 a_{x,y}\in\H\subset\R,
$$
where the family of elements $(a_{x,y})_{y\in X}$ converges to zero
in $\R$ for every fixed $x\in X$.
 Now we can compute that
\begin{multline*}
 0 = q - qe = \sum\nolimits_{x\in X}q_xx - \sum\nolimits_{x\in X}q_xxe \\
 = \sum\nolimits_{x\in X}q_x\sum\nolimits_{y\in Y}\delta_{x,y}y -
 \sum\nolimits_{x\in X}q_x\sum\nolimits_{y\in Y}a_{x,y}y \\
 = \sum\nolimits_{x\in X}q_x\sum\nolimits_{y\in Y}
 (\delta_{x,y}-a_{x,y})y.
\end{multline*}
 The resulting equation means that
\begin{equation} \label{matrix-equation}
 \sum\nolimits_{x\in X}q_x(\delta_{x,y}-a_{x,y})=0
 \qquad\text{for every $y\in X$}.
\end{equation}

 If the set $X$ is empty, then there is nothing to prove.
 Otherwise, choose an element $x_0\in X$, and consider the $X\times X$
matrix $Q=(q_{z,x})_{z,x\in X}$ with the entries $q_{z,x}=q_x$
when $z=x_0$ and $q_{z,x}=0$ when $z\ne x_0$.
 In other words, we consider the family of elements $(q_x)_{x\in X}$ as
an $X$\+indexed row and build an $X\times X$ matrix in which this row
is the only nonzero one.
 We also consider the $X\times X$ matrices $A=(a_{x,y})_{x,y\in X}$
and $1=(\delta_{x,y})_{x,y\in X}$.
 All the three matrices $Q$, $A$, and~$1$ have entries in $\R$ and
zero-convergent rows, so they belong to $\Mat_X(\R)$; and, of course,
$1$~is the unit element of the ring $\Mat_X(\R)$. 
 Then the family of equations~\eqref{matrix-equation} can be expressed
as a matrix multiplication equation $Q(1-A)=0$ in the ring $\Mat_X(\R)$.

 Furthermore, by Lemma~\ref{matrix-Jacobson}, the matrix $A$ belongs
to the topological Jacobson radical of the ring $\Mat_X(\R)$.
 By Lemma~\ref{Jacobson-additionally-characterized}, the right
multiplication with $1-A$ acts injectively in $\Mat_X(\R)$.
 Thus $Q=0$, and it follows that $q=0$, as desired.
\end{proof}

\begin{proof}[Proof of Theorem~\ref{1-strictly-flat-projdim1-cover-thm}]
 Let $\F$ be a $1$\+strictly flat left $\R$\+contramodule of
projective dimension not exceeding~$1$, and $p\:\P\rarrow\F$ be its
projective cover in $\R\contra$.
 Since there are enough projective objects in $\R\contra$, the map~$p$
is surjective.
 Put $\K=\ker(p)$.
 Then $\K$ is a projective $\R$\+contramodule.
 By~\cite[Lemma~3.1]{Pproperf}, $\K$ is a superfluous
subcontramodule in~$\P$.

 According to Proposition~\ref{superfluous-subcontramodule}, we have
$\K\subset\H\tim\P$.
 From Proposition~\ref{1-strictly-flat-projdim1-prop}, we know that
$\K\cap(\H\tim\P)=\H\tim\K$.
 Thus $\K=\H\tim\K$.
 By Proposition~\ref{projective-contra-Nakayama}, it follows that $\K=0$.
 We can conclude that $\F\simeq\P$ is a projective left
$\R$\+contramodule.
\end{proof}

 In particular, by~\cite[Corollary~2.4]{Pproperf}, all countable direct
limits of projective contramodules are $1$\+strictly flat (in fact,
$\infty$\+strictly flat) of projective dimension not exceeding~$1$.
 Hence it follows from Theorem~\ref{1-strictly-flat-projdim1-cover-thm}
that a countable direct limit of projective left $\R$\+contramodules
is projective if it has a projective cover.
 The following special case is of interest.

 A \emph{Bass flat contramodule} is a countable direct limit of free
left $\R$\+contramodules with one generator, computed in the category
of left $\R$\+contramodules $\R\contra$,
$$
 \B=\varinjlim\nolimits^{\R\contra}_\omega
 \,(\R\overset{*a_1}\lrarrow\R\overset{*a_2}\lrarrow\R
 \overset{*a_3}\lrarrow\dotsb)
$$
where $a_1$, $a_2$,~\dots\ is a sequence of elements of $\R$ and
$*a\:\R\rarrow\R$ is the left $\R$\+contramodule morphism of right
multiplication with $a\in\R$.
 
\begin{cor} \label{bass-flat-contra-cover}
 If a Bass flat left\/ $\R$\+contramodule\/ $\B$ has a projective
cover in\/ $\R\contra$, then\/ $\B$ is projective (as a left\/
$\R$\+contramodule). \qed
\end{cor}

 The generalization to uncountable direct limits of projective
contramodules will be obtained as
Corollary~\ref{direct-limits-proj-contramodules-no-proj-covers}
in Section~\ref{contramodule-direct-limits-secn}.

\Section{Quasi-Split Exact Sequences and Locally Split Morphisms}

 In the next several sections, we consider two kinds of abelian
categories.
 An abelian category is said to be \emph{Ab3} if it is cocomplete,
or in other words, if it has set-indexed coproducts.
 A cocomplete abelian category is said to satisfy \emph{Ab5} if
it has exact direct limit functors.

 In particular, the category of left contramodules $\R\contra$ over
a topological ring $\R$ is Ab3, but usually not~Ab5.
 A \emph{Grothendieck abelian category} is an Ab5~category with
a generator.
 The category of left modules $A\modl$ over any associative ring $A$
is Grothendieck.

 Let $\sA$ be an Ab5~category.
 We will say that a short exact sequence $0\rarrow K\rarrow C\rarrow
D\rarrow0$ in $\sA$ is \emph{quasi-split} if it is the direct limit
of a direct system of split short exact sequences $0\rarrow K_x\rarrow C
\rarrow D_x\rarrow0$, indexed by a directed poset $X$, where
$(K_x)_{x\in X}$ and $(D_x)_{x\in X}$ are direct systems of objects
in $\sA$ and $(C_x=C)_{x\in X}$ is a constant direct system.
 If this is the case, the morphism $K\rarrow C$ is said to be
a \emph{quasi-split monomorphism} and the morphism $C\rarrow D$ is
a \emph{quasi-split epimorphism}.

 The definition of a quasi-split short exact sequence in
an Ab3~category $\sB$ is slightly more complicated.
 Suppose that we are given a direct system
\begin{equation} \label{direct-system-constant-middle}
 0\lrarrow\K_x\lrarrow\C\lrarrow\D_x\lrarrow0
\end{equation}
of split short exact sequences in $\sB$, indexed by a directed poset
$X$, where $(\K_x)_{x\in X}$ and $(\D_x)_{x\in X}$ are direct systems
of objects and $(\C_x=\C)_{x\in X}$ is a constant direct system.

 The direct limit of a direct system of short exact sequences in
$\sB$ does not need to be exact, but only right exact; so
the direct limit of~\eqref{direct-system-constant-middle} is a right
exact sequence
\begin{equation} \label{right-exact-sequence}
 \fM\lrarrow\C\lrarrow\D\lrarrow0.
\end{equation}
 Let $\K$ be the image of the morphism $\fM\rarrow\C$.
 Then we will say that the short exact sequence $0\rarrow\K\rarrow
\C\rarrow\D\rarrow0$ in $\sB$ is \emph{quasi-split}, the morphism
$\K\rarrow\C$ is a \emph{quasi-split monomorphism}, and the morphism
$\C\rarrow\D$ is a \emph{quasi-split epimorphism}.

 The following proposition (generalizing~\cite[Lemma~2.1]{GG})
introduces the class of examples we are mainly interested in.

\begin{prop} \label{direct-limit-epi-quasi-split}
 Let $(\fc_{y,x}\:\C_x\to\C_y)_{x<y\in X}$ be a diagram in a cocomplete
abelian category\/ $\sB$, indexed by a directed poset $X$, and let\/
$\D=\varinjlim^\sB_{x\in X}\C_x$ be its direct limit in\/~$\sB$.
 Then the natural epimorphism\/ $\fp\:\C'=\coprod^\sB_{x\in X}\C_x
\rarrow\D$ is quasi-split.
\end{prop}

\begin{proof}
 We have a natural right exact sequence
\begin{equation} \label{bar-direct-limit}
 \coprod_{x<y\in X}\C_{x,y}\overset\ft\lrarrow\coprod_{x\in X}\C_x
 \overset\fp\lrarrow\D\lrarrow 0,
\end{equation}
where $\C_{x,y}=\C_x$ is a copy of the object $\C_x$ for every pair
of elements $x<y$ in~$X$.

 Fix an element $z\in X$, and consider the subdiagram
$(\C_x)_{x\in X}^{x\le z}$ of our diagram $(\C_x)_{x\in X}$ formed
by all the objects $\C_x$ with $x\le z$ and the morphisms
$\fc_{y,x}\:\C_x\rarrow\C_y$, \ $x<y\le z$.
 Obviously, the object $\C_z$ is the direct limit of the diagram
$(\C_x)_{x\in X}^{x\le z}$.
 So the right exact sequence~\eqref{bar-direct-limit} for the diagram
$(\C_x)_{x\in X}^{x\le z}$ takes the form
\begin{equation} \label{bar-bounded-subdiagram}
 \coprod_{x<y\in X}^{y\le z}\C_{x,y}\overset{\ft_z}\lrarrow
 \coprod_{x\in X}^{x\le z}\C_x\overset{\fp_z}\lrarrow\C_z\lrarrow0.
\end{equation}
 As the element $z\in X$ varies, the right exact
sequences~\eqref{bar-bounded-subdiagram} form a diagram indexed by
the same poset~$X$.
 Given a pair of elements $z<w\in X$, the related morphisms on
the middle and the leftmost terms of~\eqref{bar-bounded-subdiagram}
are the subcoproduct inclusions corresponding to the inclusions of
subsets $\{x\in X\mid x\le z\}\hookrightarrow\{x\in X\mid x\le w\}$
and $\{(x,y)\mid x<y\le z\}\hookrightarrow\{(x,y)\mid x<y\le w\}$,
while the morphism on the rightmost terms is
$\fc_{w,z}\:\C_z\rarrow\C_w$.
 The right exact sequence~\eqref{bar-direct-limit} is the direct limit
of the right exact sequences~\eqref{bar-bounded-subdiagram}
over $z\in X$.

 We denote by $\K$ the image of the morphism~$\ft$ (which coincides
with the kernel of the morphism~$\fp$) in
the sequence~\eqref{bar-direct-limit}.
 Similarly, let us denote by $\K_z$ the image of the morphism~$\ft_z$
(which coincides with the kernel of the morphism~$\fp_z$) in
the sequence~\eqref{bar-bounded-subdiagram}.
 So we have short exact sequences
\begin{equation} \label{generalized-telescope}
 0\lrarrow\K\overset\ii\lrarrow\coprod_{x\in X}\C_x
 \overset\fp\lrarrow\D\lrarrow 0
\end{equation}
and
\begin{equation} \label{bounded-generalized-telescope}
 0\lrarrow\K_z\overset{\ii_z}\lrarrow \coprod_{x\in X}^{x\le z}\C_x
 \overset{\fp_z}\lrarrow\C_z\lrarrow 0.
\end{equation}
 As the element $z\in X$ varies, the short exact
sequences~\eqref{bounded-generalized-telescope} form a diagram, indexed
by the poset~$X$.
 The morphism~$\fp$ is the direct limit of the morphisms~$\fp_z$, while
the object $\K$ does \emph{not} need to be the direct limit of
the objects $\K_z$ (as the direct limits in $\sB$ are only right exact).
 In fact, the direct limit of the short exact
sequences~\eqref{bounded-generalized-telescope} is, generally speaking,
a right exact sequence of the form
\begin{equation} \label{pseudo-telescope}
 \fM=\varinjlim_{z\in X}\K_z\overset\fm\lrarrow
 \coprod_{x\in X}\C_x\overset\fp\lrarrow\D\lrarrow0,
\end{equation}
and the object $\K$ is the image of the morphism~$\fm$.

 Notice that the object $\K_z$ is, of course, naturally isomorphic
to $\coprod_{x\in X}^{x<z}\C_x$, and the short exact
sequence~\eqref{bounded-generalized-telescope} is naturally split.
 However, the morphism
$\ii_z\:\coprod_{x\in X}^{x<z}\C_x\rarrow\coprod_{x\in X}^{x\le z}\C_z$
is \emph{not} the subcoproduct inclusion related to the inclusion
of subsets $\{x\mid x<z\}\hookrightarrow\{x\mid x\le z\}$.
 Rather, it is a certain ``diagonal'' map which can be constructed
in terms of the morphisms $\fc_{z,x}\:\C_x\rarrow\C_z$.
 Given a pair of elements $z<w$, the related morphism between
the middle terms of the sequences~\eqref{bounded-generalized-telescope}
is the subcoproduct inclusion described above, but the related
morphism $\fk_{w,z}\:\K_z\rarrow\K_w$ between the leftmost terms of
the sequences~\eqref{bounded-generalized-telescope} is \emph{not}
the subcoproduct inclusion.

 Now we recall the notation $\C'=\coprod_{x\in X}\C_x$, and set $\D_z$
to be the cokernel of the composition of split monomorphisms
$\K_z\rarrow\coprod_{x\in X}^{x\le z}\C_x\rarrow\C'$.
 Then we have a direct system of split short exact sequences
\begin{equation} \label{direct-limit-related-constant-middle}
 0\lrarrow\K_z\overset{\ii_z'}\lrarrow\C'\overset{\fp'_z}\lrarrow\D_z
 \lrarrow0,
\end{equation}
where the objects $(\C_z'=\C')_{z\in X}$ form a constant direct system
and the morphism~$\fm$ is the direct limit of the morphisms~$\ii_z'$
over $z\in X$.
 Hence the morphism~$\fp$ is the direct limit of the morphisms~$\fp_z'$,
and we are done.
\end{proof}

 The next concept is more general than quasi-splitness.
 Before introducing it for arbitrary Ab3~categories, let us recall its
definition for the categories of modules.
 Let $A$ be an associative ring, and $K\subset C$ be (say, left)
$A$\+modules.
 Then the submodule $K$ is said to be \emph{locally split} in $C$
if for every element $k\in K$ there exists an $A$\+module morphism
$h\:C\rarrow K$ such that $h(k)=k$.
 If this is the case, then for any finite set of elements
$k_1$,~\dots, $k_n\in K$ there exists an $A$\+module morphism
$g\:C\rarrow K$ such that $g(k_i)=k_i$ for all $i=1$,~\dots,~$n$
\cite[Proposition~1.2\,(2)\,$\Rightarrow$\,(1)]{Zim}.

 Let $\fm\:\fM\rarrow\C$ be a morphism in an Ab3~category~$\sB$.
 We will say that $\fm$~is \emph{locally split} if there exists
a direct system $(\K_x)_{x\in X}$ in the category $\sB$, indexed by some
directed poset~$X$, an epimorphism $\fs\:\varinjlim_{x\in X}\K_x
\rarrow\fM$, and morphisms $\fg_x\:\C\rarrow\fM$, \,$x\in X$, such that
the equation $\fg_x\fm\fs\fk_x=\fs\fk_x$ holds for every $x\in X$,
where $\fk_x\:\K_x\rarrow\varinjlim_{y\in X}\K_y$ is the canonical
morphism.
 A subobject $\K\subset\C$ is said to be \emph{locally split}
if its inclusion morphism $\ii\:\K\rarrow\C$ is locally split.

 Let us emphasize that our definition of a locally split morphism is
designed to handle locally split \emph{monomorphisms}.
 It is \emph{not} relevant to the notion of a locally split epimorphism
(which has been also considered in the module theory literature, as per
the references in the introduction; for example, a dual version of
the above result about the equivalence of the local splitness for one
element and for a finite number of elements can be found
in~\cite[Corollary~2]{Az2}).

 Though a locally split morphism in an Ab3 category, in the sense of
our definition, does not need to be a monomorphism, any locally split
morphism in an Ab5 category is a monomorphism, as we will see in
the next section.

\begin{lem} \label{image-of-locally-split}
 In any Ab3 category, the image of a locally split morphism is
a locally split subobject.
\end{lem}

\begin{proof}
 Denote by $\fN$ the image of a morphism $\fm\:\fM\rarrow\C$;
so $\fm$~decomposes as $\fM\overset\fn\rarrow\fN\overset\ii\rarrow\C$,
where $\fn$~is an epimorphism and $\ii$~is a monomorphism.
 Assuming that the morphism~$\fm$ is locally split, we have to
show that the morphism~$\ii$ is.
 Indeed, let $(\K_x)_{x\in X}$ be a direct system, $\fs\:
\varinjlim_{x\in X}\K_x\rarrow\fM$ be an epimorphism, and
$\fg_x\:\C\rarrow\fM$ be a family of morphisms witnessing
the local splitness of the morphism~$\fm$.
 Then the same direct system $(\K_x)_{x\in X}$, the epimorphism
$\fn\fs\:\varinjlim_{x\in X}\K_x\rarrow\fN$, and the morphisms
$\fn\fg_x\:\C\rarrow\fN$ witness the local splitness of
the monomorphism~$\ii$ (because the equations $\fg_x\fm\fs\fk_x
=\fs\fk_x$ imply the equations $\fn\fg_x\ii\fn\fs\fk_x=\fn\fs\fk_x$
for all $x\in X$).
\end{proof}

 The following lemma shows that our terminology is consistent with
the classical definition for module categories.

\begin{lem} \label{local-splitness-classical-and-categorical}
 Let $A$ be an associative ring and $i\:K\rarrow C$ be an injective
morphism of left $A$\+modules.
 Then the morphism~$i$ is locally split in $A\modl$, in the sense of
the above categorical definition, if and only if it is locally
split in the classical sense.
\end{lem}

\begin{proof}
 ``If'': assume that for every finitely generated submodule
$M\subset K$ there exists a morphism $g\:C\rarrow K$ such that
$gi(m)=m$ for all $m\in M$.
 Let $X$ denote the poset of all finitely generated submodules
in $K$, ordered by inclusion, and let $K_x\subset K$ be
the finitely generated submodule corresponding to an element $x\in X$.
 Let $s\:\varinjlim_{x\in X}K_x\rarrow K$ be the natural isomorphism.
 For every $x\in X$, let $g_x\:C\rarrow K$ be a morphism such that
$g_xi(k)=k$ for all $k\in K_x$.
 Then the direct system $(K_x)_{x\in X}$, the isomorphism
$s\:\varinjlim_{x\in X}K_x\rarrow K$, and the morphisms
$g_x\:C\rarrow K$ witness the local splitness of the monomorphism
$i\:K\rarrow C$.

 ``Only if'': assume that a direct system $(K_x)_{x\in X}$,
an epimorphism $s\:\varinjlim_{x\in X}K_x\rarrow K$, and some
morphisms $g_x\:C\rarrow K$ witness the local splitness of
the monomorphism $i\:K\rarrow C$.
 Then, for any finitely generated submodule $M\subset K$, there
exists an index $x\in X$ such that the submodule $M\subset K$ is
contained in the image of the composition $K_x\overset{k_x}\rarrow
\varinjlim_{y\in X}K_y\overset{s}\rarrow K$.
 It follows that the morphism $g_x\:C\rarrow K$ satisfies
the equation $g_xi(m)=m$ for all $m\in M$.
\end{proof}

 The proof of Lemma~\ref{local-splitness-classical-and-categorical}
shows that, for direct systems of \emph{finitely generated} modules
$(K_x)_{x\in X}$ with an epimorphism $s\:\varinjlim_{x\in X}K_x
\rarrow K$, the possibility to satisfy the definition of local
splitness by finding a suitable family of morphisms $g_x\:C\rarrow K$
does \emph{not} depend on the choice of a particular direct system.
 If the local splitness of an injective $A$\+module morphism
$i\:K\rarrow C$ is witnessed by a direct system $(K_x)_{x\in X}$,
an epimorphism $s\:\varinjlim_{x\in X}K_x\rarrow K$, and some morphisms
$g_x\:C\rarrow K$, \,$x\in X$, then for any direct system of finitely
generated $A$\+modules $(L_y)_{y\in Y}$ and an epimorphism
$t\:\varinjlim_{y\in Y}L_y\rarrow K$ one can find
morphisms $h_y\:C\rarrow K$, \,$y\in Y$ which also witness
the local splitness of~$i$.

\begin{lem} \label{quasi-split-is-locally-split}
 Any quasi-split monomorphism in an Ab3~category is locally split.
\end{lem}

\begin{proof}
 Assume that we are given a direct system of split short exact
sequences~\eqref{direct-system-constant-middle}, whose direct limit
is a right exact sequence~\eqref{right-exact-sequence}, and $\K$
is the image of the morphism $\fm\:\fM\rarrow\C$.
 Let $\fk_x\:\K_x\rarrow\fM$ denote the canonical morphism
$\K_x\rarrow\varinjlim_{y\in X}\K_y$, and let $\fh_x\:\C\rarrow\K_x$
be a morphism splitting the short exact sequence $0\rarrow\K_x
\overset{\fm\fk_x}\lrarrow\C\rarrow\D_x\rarrow0$
\eqref{direct-system-constant-middle}.
 Then the direct system $(\K_x)_{x\in X}$, the identity isomorphism
$\varinjlim_{x\in X}\K_x\rarrow\fM$, and the morphisms
$\fg_x=\fk_x\fh_x\:\C\rarrow\fM$ witness the local splitness of
the morphism $\fm\:\fM\rarrow\C$.
 Indeed, we have $\fh_x\fm\fk_x=\id_{\K_x}$ for every $x\in X$,
hence $\fg_x\fm\fk_x=\fk_x\fh_x\fm\fk_x=\fk_x$.
 By Lemma~\ref{image-of-locally-split}, the image $\K$ of a locally
split morphism $\fm\:\fM\rarrow\C$ is a locally split subobject
in $\C$, as desired.
\end{proof}

\begin{lem} \label{R-preserves-quasi-split-locally-split}
 Let\/ $\sA$ and\/ $\sB$ be Ab3~categories, and let $R\:\sB\rarrow\sA$
be a colimit-preserving functor.
 Then \par
\textup{(a)} $R$ takes quasi-split epimorphisms in\/ $\sB$
to quasi-split epimorphisms in\/~$\sA$; \par
\textup{(b)} $R$ takes locally split morphisms in\/ $\sB$ to
locally split morphisms in\/~$\sA$.
\end{lem}

\begin{proof}
 Both the assertions follow immediately from the definitions.
\end{proof}

 The next lemma explains the relevance of the local splitness
property to covers.

\begin{lem} \label{cover-locally-split-kernel}
 Let\/ $\sB$ be an Ab3~category and\/ $\sC\subset\sB$ be a class of
objects.
 Then any\/ $\sC$\+cover with a locally split kernel is
a monomorphism in\/~$\sB$.
 In particular, if an epimorphism with a locally split kernel is
a cover, then it is an isomorphism.
\end{lem}

\begin{proof}
 Let $\fq\:\Q\rarrow\B$ be a $\sC$\+cover of an object $\B\in\sB$,
and let $\fj\:\L\rarrow\Q$ be the kernel of~$\fq$.
 Assume that a direct system $(\L_x)_{x\in X}$, an epimorphism
$\ft\:\varinjlim_{x\in X}\L_x\rarrow\L$, and some morphisms
$\fh_x\:\Q\rarrow\L$ witness the local splitness of~$\fj$.
 Let $\fl_x\:\L_x\rarrow\varinjlim_{y\in X}\L_y$ be the canonical
morphism.
 Then we have $\fh_x\fj\ft\fl_x=\ft\fl_x$ for all $x\in X$.
 
 Consider the endomorphism $(\id_\Q-\fj\fh_x)\:\Q\rarrow\Q$.
 We have $\fq(\id_\Q-\fj\fh_x)=\fq$, since $\fq\fj=0$.
 Since $\fq$~is a cover, it follows that $\id_\Q-\fj\fh_x$ is
an automorphism of~$\Q$.
 Now the equations
\begin{equation} \label{locally-split-cover-computation}
 (\id_\Q-\fj\fh_x)\fj\ft\fl_x=
 \fj\ft\fl_x-\fj\fh_x\fj\ft\fl_x=
 \fj\ft\fl_x-\fj\ft\fl_x=0
\end{equation}
imply $\fj\ft\fl_x=0$.
 Since $\fj$~is a monomorphism, it follows that $\ft\fl_x=0$.
 As this holds for all $x\in X$, we can conclude that $\ft=0$.
 Since $\ft$~is an epimorphism by assumption, this means that $\L=0$;
so $\fq$~is a monomorphism.
\end{proof}

\Section{Locally Split Morphisms and Covers in Ab5 Categories}

 In this section, $\sA$ denotes a cocomplete abelian category with
exact direct limits (i.~e., an Ab5~category).

\begin{lem} \label{any-locally-split-is-mono}
 In an Ab5~category, any locally split morphism is a monomorphism.
\end{lem}

\begin{proof}
 Let $m\:M\rarrow C$ be a morphism in~$\sA$.
 Assume that a direct system $(K_x)_{x\in X}$, an epimorphism
$s\:\varinjlim_{x\in X}K_x\rarrow M$, and some morphisms
$g_x\:C\rarrow M$ witness the local splitness of~$m$.
 Let $k_x\:K_x\rarrow\varinjlim_{y\in X}K_y$ be the canonical morphism.
 Then the equation $g_xmsk_x=sk_x$ holds for every $x\in X$.
 
 Suppose that $\rho\:L\rarrow M$ is a morphism in~$\sA$ such that
$m\rho=0$.
 Denote by $L_x=L\sqcap_M K_x$ the fibered product of the pair
of morphisms $\rho\:L\rarrow M$ and $sk_x\:K_x\rarrow M$.
 In an Ab5 category, fibered products commute with direct limits; so
we have $\varinjlim_{x\in X}L_x=L\sqcap_M\varinjlim_{x\in X}K_x$.
 Since the morphism $s\:\varinjlim_{x\in X}K_x\rarrow M$ is
an epimorphism, it follows that the natural morphism
$t\:\varinjlim_{x\in X}L_x\rarrow L$ is an epimorphism, too.

 Let $l_x\:L_x\rarrow\varinjlim_{y\in X}L_y$ be the canonical
morphism.
 Then the canonical morphism $L_x\rarrow L$ decomposes as $L_x
\overset{l_x}\rarrow\varinjlim_{y\in X}L_y\overset t\rarrow L$.
 Denote the canonical morphism $L_x\rarrow K_x$ by~$\rho_x$.
 Then the diagram
$$
\begin{diagram}
\node{L_x}\arrow{s,l}{\rho_x}\arrow[2]{e,t}{l_x}
\node[2]{\varinjlim\nolimits_{y\in Y}L_y}
\arrow{s}\arrow[2]{e,t,A}{t} \node[2]{L}\arrow{s,r}{\rho} \\
\node{K_x}\arrow[2]{e,t}{k_x}\node[2]{\varinjlim\nolimits_{y\in Y}K_y}
\arrow[2]{e,t,A}{s} \node[2]{M}
\end{diagram}
$$
is commutative, so we have $sk_x\rho_x=\rho tl_x\:L_x
\rarrow M$ for every $x\in X$.
 Now we can compute that
$$
 \rho tl_x=sk_x\rho_x=g_xmsk_x\rho_x=g_xm\rho tl_x=0,
$$
since $m\rho=0$.
 As this holds for all $x\in X$, it follows that $\rho t=0$.
 Since $t$~is an epimorphism, we can conclude that $\rho=0$.
 Thus $m$~is a monomorphism.
\end{proof}

\begin{lem} \label{direct-summand-of-locally-split}
 In an Ab5~category, any direct summand of a locally split
monomorphism is a locally split monomorphism.
\end{lem}

\begin{proof}
 Let $i\:K\rarrow C$ be a monomorphism in~$\sA$.
 Assume that a direct system $(K_x)_{x\in X}$, an epimorphism
$s\:\varinjlim_{x\in X}K_x\rarrow K$, and some morphisms
$g_x\:C\rarrow K$ witness the local splitness of~$i$.
 Let $k_x\:K_x\rarrow\varinjlim_{y\in X}K_y$ be the canonical morphism.
 Then the equation $g_xisk_x=sk_x$ holds for every $x\in X$.
 
 Suppose that a (mono)morphism $j\:L\rarrow Q$ in $\sA$ is a direct
summand of the monomorphism~$i$.
 Then we have a commutative diagram
\begin{equation} \label{direct-summand-morphism}
\begin{diagram}
\node{L}\arrow{e,t}{j}\arrow{s,l}{\rho}
\node{Q}\arrow{s,l}{\gamma} \\
\node{K}\arrow{e,t}{i}\arrow{s,l}{\lambda}
\node{C}\arrow{s,l}{\beta} \\
\node{L}\arrow{e,t}{j} \node{Q}
\end{diagram}
\end{equation}
where both the vertical compositions are identity morphisms.

 Denote by $L_x=L\sqcap_KK_x$ the fibered product of the pair of
morphisms $\rho\:L\rarrow K$ and $sk_x\:K_x\rarrow K$.
 As in the previous proof, we have $\varinjlim_{x\in X}L_x=
L\sqcap_K\varinjlim_{x\in X}K_x$, since $\sA$ is an Ab5~category.
 The natural morphism $t\:\varinjlim_{x\in X}L_x\rarrow L$ is
an epimorphism, since the morphism~$s$ is.

 Following the notation of the previous proof, the canonical
morphism $L_x\rarrow L$ decomposes as $L_x\overset{l_x}\rarrow
\varinjlim_{y\in X}L_y\overset t\rarrow L$.
 Denote the canonical morphism $L_x\rarrow K_x$ by~$\rho_x$.
 Similarly to the previous proof, we have $sk_x\rho_x=\rho tl_x\:
L_x\rarrow K$.

 For every $x\in X$, denote by $h_x\:Q\rarrow L$ the composition
\begin{equation} \label{producing-h-out-of-g}
 Q\overset\gamma\lrarrow C\overset{g_x}\lrarrow K\overset\lambda
 \lrarrow L.
\end{equation}
 We claim that the direct system $(L_x)_{x\in X}$,
the epimorphism $t\:\varinjlim_{x\in X}L_x\rarrow L$, and
the morphisms $h_x\:Q\rarrow L$ witness the local splitness of
the morphism~$j$.
 Indeed,
\begin{equation} \label{direct-summand-witnessing-computation}
 h_xjtl_x=\lambda g_x\gamma jtl_x=\lambda g_xi\rho tl_x
 =\lambda g_xisk_x\rho_x=\lambda sk_x\rho_x=
 \lambda\rho tl_x=tl_x,
\end{equation}
since $\gamma j=i\rho$ and $\lambda\rho=\id_L$.
\end{proof}

\begin{cor} \label{ab5-locally-split-precover-and-a-cover-cor}
 Let\/ $\sA$ be an Ab5~category and\/ $\sC\subset\sA$ be a class
of objects.
 Let $p\:C\rarrow D$ be an epimorphism in\/~$\sA$.
 Assume that
\begin{enumerate}
\item the morphism~$p$ is a\/ $\sC$\+precover with a locally split
kernel; \par
\item the object $D\in\sA$ has a\/ $\sC$\+cover.
\end{enumerate}
 Then $D\in\sC$ and the epimorphism~$p$ is split.
\end{cor}

\begin{proof}
 Let $q\:Q\rarrow D$ be a $\sC$\+cover of~$D$.
 Since $C\in\sC$ and $p$~is an epimorphism, the morphism~$q$ is
an epimorphism, too.
 Let $i\:K\rarrow C$ and $j\:L\rarrow Q$ be the kernels of
the morphisms $p$ and~$q$, respectively.
 Since $p$~is a $\sC$\+precover and $q$~is a $\sC$\+cover,
the short exact sequence $0\rarrow L\overset j\rarrow Q\overset
q\rarrow D\rarrow0$ is a direct summand of the short exact sequence
$0\rarrow K\overset i\rarrow C\overset p\rarrow D\rarrow0$.

 In particular, the monomorphism~$j$ is a direct summand of
the monomorphism~$i$.
 By assumption, the morphism~$i$ is locally split.
 Applying Lemma~\ref{direct-summand-of-locally-split}, we conclude
that the morphism~$j$ is locally split.
 By Lemma~\ref{cover-locally-split-kernel}, it follows that $L=0$
and $q$~is an isomorphism.
 Consequently, $D\in\sC$ and the epimorphism~$p$ is split.
\end{proof}

\begin{thm} \label{ab5-direct-limit-precover-cover-thm}
 Let\/ $\sA$ be an Ab5~category and\/ $\sC\subset\sA$ be a class of
objects closed under coproducts and direct summands.
 Let $(c_{y,x}\:C_x\to C_y)_{x<y\in X}$ be a diagram of objects
$C_x\in\sC$, indexed by a directed poset $X$, and
let $D=\varinjlim_{x\in X}^\sA C_x$ be its direct limit in
the category\/~$\sA$.
 Let $p\:\coprod_{x\in X}C_x\rarrow D$ be the natural epimorphism.
 Assume that
\begin{enumerate}
\item the morphism~$p$ is a\/ $\sC$\+precover in\/~$\sA$;
\item the object $D$ has a\/ $\sC$\+cover in\/~$\sA$.
\end{enumerate}
 Then $D\in\sC$ and the epimorphism~$p$ is split.
\end{thm}

\begin{proof}
 By Proposition~\ref{direct-limit-epi-quasi-split}, the epimorphism~$p$
is quasi-split; so its kernel~$i$ is a quasi-split monomorphism.
 According to Lemma~\ref{quasi-split-is-locally-split}, it follows that
the morphism~$i$ is locally split.
 Hence Corollary~\ref{ab5-locally-split-precover-and-a-cover-cor} is
applicable, and we are done.
\end{proof}

\Section{Locally Split Morphisms and Covers in Ab3 Categories}
\label{ab3-secn}

 In this section we consider an Ab3~category $\sB$, an Ab5~category
$\sA$, and a functor $R\:\sB\rarrow\sA$ preserving all colimits.
 Equivalently, $R$ is a right exact functor preserving coproducts.
 Any such functor is additive.
 We will denote the functor $R$ by $\C\longmapsto\ov\C$ for brevity.

 Following~\cite[Section~8]{BP3}, we say that a short exact sequence
$0\rarrow\K\rarrow\C\rarrow\D\rarrow0$ in the category $\sB$ is
\emph{functor pure} (or \emph{f\+pure} for brevity) if, for any
Ab5~category $\sA$ and any colimit-preserving functor
$R\:\sB\rarrow\sA$, the short sequence $0\rarrow R(\K)\rarrow R(\C)
\rarrow R(\D)\rarrow0$ is exact in~$\sA$.
 Equivalently, this means that the morphism $\ov\K\rarrow\ov\C$ is
a monomorphism.

 If a short exact sequence $0\rarrow\K\rarrow\C\rarrow\D\rarrow0$
is f\+pure, we will say that $\K\rarrow\C$ is
an \emph{f\+pure monomorphism} and $\C\rarrow\D$ is
an \emph{f\+pure epimorphism}.

\begin{lem} \label{direct-limit-f-pure-epi}
 For any cocomplete abelian category\/ $\sB$, the class of all
f\+pure epimorphisms in\/ $\sB$ is closed under direct limits.
\end{lem}

\begin{proof}
 The class of all epimorphisms is closed under all colimits in any
cocomplete category.
 In our context, we need to prove a similar property for the direct
limits of f\+pure epimorphisms in~$\sB$.

 Let $X$ be a directed poset, and let
$$
(\fp_x)_{x\in X}\:(\fb_{y,x}\:\B_x\to\B_y)_{x<y\in X}\lrarrow
(\fc_{y,x}\:\C_x\to\C_y)_{x<y\in X}
$$
be a morphism of $X$\+indexed diagrams in $\sB$ such that the morphism
$\fp_x\:\B_x\rarrow\C_x$ is an f\+pure epimorphism for every index
$x\in X$.
 Put $\K_x=\ker(\fp_x)$; so we have a short exact sequence of
diagrams
\begin{equation} \label{contramodule-short-exact-sequence-of-diagrams}
 0\lrarrow(\K_x)_{x\in X}\overset{(\ii_x)}\lrarrow(\B_x)_{x\in X}
 \overset{(\fp_x)}\lrarrow(\C_x)_{x\in X}\lrarrow0
\end{equation}
such that the short exact sequence of objects $0\rarrow\K_x\rarrow
\B_x\rarrow\C_x\rarrow0$ is f\+pure exact in $\sB$ for every $x\in X$.

 In any cocomplete abelian category, all colimit functors are right
exact.
 Passing to the direct limit
of~\eqref{contramodule-short-exact-sequence-of-diagrams}, we obtain
a right exact sequence
\begin{equation} \label{contramodule-direct-limit-sequence}
 \fM\overset\fm\lrarrow\B\overset\fp\lrarrow\C\lrarrow0,
\end{equation}
where $\fM=\varinjlim_{x\in X}\K_x$, \ $\B=\varinjlim_{x\in X}\B_x$,
and $\C=\varinjlim_{x\in X}\C_x$.
 Denote by $\ii\:\K\rarrow\B$ the kernel of the epimorphism~$\fp$.
 Then the object $\K$ is also the image of the morphism~$\fm$,
which factorizes into the compostion $\fm=\ii\fn$ of an epimorphism
$\fn\:\fM\rarrow\K$ and the monomorphism $\ii\:\K\rarrow\B$.

 We need to prove that the morphism $R(\ii)=\bar\ii\:\ov\K\rarrow\ov\B$
is a monomorphism in~$\sA$.
 By assumption, the functor $R$ preserves exactness of the short
sequences~\eqref{contramodule-short-exact-sequence-of-diagrams};
so we get a short exact sequence of diagrams
\begin{equation} \label{short-exact-sequence-of-diagrams}
 0\lrarrow(\ov\K_x)_{x\in X}\overset{(\bar\ii_x)}\lrarrow
 (\ov\B_x)_{x\in X}\overset{(\bar\fp_x)}\lrarrow(\ov\C_x)_{x\in X}
 \lrarrow0
\end{equation}
in the category~$\sA$.
 Direct limits are exact in $\sA$; so passing to the direct limit
of~\eqref{short-exact-sequence-of-diagrams} we obtain a short exact
sequence
\begin{equation} \label{direct-limit-sequence}
 0\lrarrow K\overset i\lrarrow B\overset p\lrarrow C\lrarrow0
\end{equation}
in the category~$\sA$.
 Furthermore, the functor $R$ preserves direct limits; so it takes
the right exact sequence~\eqref{contramodule-direct-limit-sequence}
to the exact sequence~\eqref{direct-limit-sequence}.

 We have shown that $K=\ov\fM$, \ $B=\ov\B$, and the morphism
$\ov\fm=i\:K\rarrow B$ is a monomorphism.
 It remains to recall that the morphism~$\fm$ decomposes as
$\fm=\ii\fn$, where $\fn\:\fM\rarrow\K$ is an epimorphism and
$\ii\:\K\rarrow\B$ is a monomorphism.
 The right exact functor $R$ takes epimorphisms to epimorphisms,
so $\bar\fn\:\ov\fM\rarrow\ov\K$ is an epimorphism.
 As the composition $\ov\fm=\bar\ii\bar\fn$ is a monomorphism, it
follows that $\bar\fn$~is an isomorphism and $\bar\ii$~is
a monomorphism.
 Thus we have $\ov\K=K$ and the functor $R$ takes the short exact
sequence
\begin{equation} \label{contramodule-direct-limit-kernel-sequence}
 0\lrarrow \K\overset\ii\lrarrow\B\overset\fp\lrarrow\C\lrarrow0
\end{equation}
in the category $\sB$ to the short exact
sequence~\eqref{direct-limit-sequence} in the category~$\sA$.
 So $\fp\:\B\rarrow\C$ is an f\+pure epimorphism.
\end{proof}

\begin{lem} \label{locally-split-f-pure}
 In any cocomplete abelian category, the cokernel of
any locally split morphism is an f\+pure epimorphism.
 Any locally split monomorphism is f\+pure.
\end{lem}

\begin{proof}
 Let $\fm\:\fM\rarrow\C$ be a locally split morphism in~$\sB$,
and let $\fp\:\C\rarrow\D$ be the cokernel of~$\fm$.
 Denote by $\K$ the image of the morphism~$\fm$.
 Then the morphism $\fm$~decomposes as $\fM\overset\fn\rarrow\K
\overset\ii\rarrow\C$, where $\fn$~is an epimorphism and
$\ii$~is a monomorphism.
 By Lemma~\ref{image-of-locally-split}, the morphism~$\ii$ is
locally split.

 Applying the functor $R$, we get the morphism $\ov\fm=\bar\ii
\bar\fn$, where $\bar\fn$ is an epimorphism.
 By Lemma~\ref{R-preserves-quasi-split-locally-split}(b), both
the morphisms $\ov\fm$ and~$\bar\ii$ are locally split in~$\sA$.
 By Lemma~\ref{any-locally-split-is-mono}, both the morphisms
$\ov\fm$ and~$\bar\ii$ are monomorphisms.

 We have shown that the short exact sequence $0\rarrow\K\overset\ii
\rarrow\C\overset\fp\rarrow\D\rarrow0$ is f\+pure in $\sB$, so
$\fp$~is an f\+pure epimorphism and $\ii$~is an f\+pure monomorphism.
 We have also shown that the morphism $\bar\fn\:\ov\fM\rarrow\ov\K$
is an isomorphism in~$\sA$.
\end{proof}

\begin{cor} \label{quasi-split-functor-pure}
 In any cocomplete abelian category, any quasi-split exact sequence
(quasi-split epimorphism, or quasi-split monomorphism) is functor pure.
\end{cor}

\begin{proof}
 By Lemma~\ref{direct-limit-f-pure-epi}, any direct limit of f\+pure
epimorphisms is an f\+pure epimorphism.
 In particular, any direct limit of split epimorphisms is an f\+pure
epimorphism.
 It follows that any quasi-split epimorphism is f\+pure.
 
 Alternatively, by Lemma~\ref{locally-split-f-pure}, any locally split
monomorphism is f\+pure.
 By Lemma~\ref{quasi-split-is-locally-split}, any quasi-split
monomorphism is locally split.
 Thus any quasi-split monomorphism is f\+pure.
\end{proof}

 The next propostion extends the result of
Corollary~\ref{ab5-locally-split-precover-and-a-cover-cor} to abelian
categories with nonexact direct limits.

\begin{prop} \label{ab3-locally-split-precover-and-cover-prop}
 Let\/ $\sB$ be an Ab3~category and\/ $\sC\subset\sB$ be a class
of objects.
 Let\/ $\fq\:\Q\rarrow\D$ be a\/ $\sC$\+cover in\/~$\sB$.
 Put\/ $\L=\ker(\fq)$.
 Assume that the object\/ $\D\in\sB$ has a\/ $\sC$\+precover with
a locally split kernel.
 Then, for any Ab5~category\/ $\sA$ and any colimit-preserving
functor $R\:\sB\rarrow\sA$, one has $R(\L)=0$.
\end{prop}

\begin{proof}
 Since any (pre)cover is a (pre)cover of its image and the images
of all $\sC$\+precovers of a given object $\D$ coincide, without loss of
generality we can replace $\D$ by $\im(\fq)$ and assume that $\fq$~is
an epimorphism.
 The argument below is a kind of conjunction of the proofs of
Lemmas~\ref{cover-locally-split-kernel}
and~\ref{direct-summand-of-locally-split}.

 Let $\fp\:\C\rarrow\D$ be a $\sC$\+precover of $\D$ with a locally
split kernel $\ii\:\K\rarrow\C$.
 Then the short exact sequence $0\rarrow\L\overset\fj\rarrow\Q\overset
\fq\rarrow\D\rarrow0$ is a direct summand of the short exact sequence
$0\rarrow\K\overset\ii\rarrow\C\overset\fp\rarrow\D\rarrow0$.
 So we have a diagram of morphisms of short exact sequences
\begin{equation} \label{precover-cover-sequences-diagram}
\begin{diagram}
\node{0}\arrow{e} \node{\L}\arrow{e,t}{\fj}\arrow{s,l}{\brho}
\node{\Q}\arrow{e,t}{\fq}\arrow{s,l}{\bgamma}
\node{\D}\arrow{e}\arrow{s,=}\node{0} \\
\node{0}\arrow{e} \node{\K}\arrow{e,t}{\ii}\arrow{s,l}{\blambda}
\node{\C}\arrow{e,t}{\fp}\arrow{s,l}{\bbeta}
\node{\D}\arrow{e}\arrow{s,=}\node{0} \\
\node{0}\arrow{e} \node{\L}\arrow{e,t}{\fj}
\node{\Q}\arrow{e,t}{\fq}
\node{\D}\arrow{e}\node{0}
\end{diagram}
\end{equation}
where all the vertical compositions are identity maps.

 Assume that a direct system $(\K_x)_{x\in X}$, an epimorphism
$\fs\:\varinjlim_{x\in X}^\sB\K_x\rarrow\K$, and some morphisms
$\fg_x\:\C\rarrow\K$ witness the local splitness of the morphism~$\ii$
in the category~$\sB$.
 As in the proof of Lemma~\ref{direct-summand-of-locally-split},
denote by $\fh_x\:\Q\rarrow\L$ the composition
$$
 \Q\overset\bgamma\lrarrow\C\overset{\fg_x}\lrarrow\K
 \overset\blambda\lrarrow\L.
$$
 As in the proof of Lemma~\ref{cover-locally-split-kernel},
consider the endomorphism $(\id_\Q-\fj\fh_x)\:\Q\rarrow\Q$.
 We have $\fq(\id_\Q-\fj\fh_x)=\fq$, since $\fq\fj=0$.
 Since $\fq$~is a cover, it follows that $\id_\Q-\fj\fh_x$ is
an automorphism of~$\Q$.

 Now we apply the functor $R$ to this whole picture.
 We will use the diagram~\eqref{direct-summand-morphism} as
a notation for the image of the leftmost and middle columns of
the diagram~\eqref{precover-cover-sequences-diagram} under
the functor~$R$.
 By Lemma~\ref{locally-split-f-pure}, the monomorphism~$\ii$ is
f\+pure; so the morphism $\bar\ii=i\:K\rarrow C$ is a monomorphism
in~$\sA$.
 It follows that the monomorphism~$\fj$ is f\+pure as well, being
a direct summand of~$\ii$; in other words, the morphism
$\bar\fj=j\:L\rarrow Q$ is a monomorphism, since it is a direct
summand of~$i$.

 Applying the functor $R$ to the direct system $(\K_x)_{x\in X}$,
we obtain a direct system $(K_x=\ov\K_x)_{x\in X}$ in
the category~$\sA$.
 The functor $R$ preserves direct limits and takes epimorphisms
to epimorphisms, so we get an epimorphism
$\bar\fs=s\:\varinjlim_{x\in X}^\sA K_x\rarrow K$.
 Let $\fk_x\:\K_x\rarrow\varinjlim_{y\in X}\K_y$ be the canonical
morphism; then $k_x=\bar\fk_x$ is the canonical morphism
$K_x\rarrow\varinjlim_{y\in X}K_y$.
 Put $\bar\fg_x=g_x\:C\rarrow K$; then the functor $R$ takes
the morphism $\fh_x=\blambda\fg_x\bgamma$ to the morphism
$h_x=\lambda g_x\gamma\:Q\rarrow L$, as
in~\eqref{producing-h-out-of-g}.

 As in the proofs of Lemmas~\ref{any-locally-split-is-mono}
and~\ref{direct-summand-of-locally-split}, denote by
$L_x=L\sqcap_KK_x$ the fibered product of the pair of morphisms
$\rho\:L\rarrow K$ and $sk_x\:K_x\rarrow K$.
 Since $\sA$ is an Ab5~category, we have $\varinjlim_{x\in X}L_x
=L\sqcap_K\varinjlim_{x\in X}K_x$, and the natural morphism
$t\:\varinjlim_{x\in X}L_x\rarrow L$ is an epimorphism.
 Let $l_x\:L_x\rarrow\varinjlim_{y\in X}L_y$ be the canonical morphism.
 Now the computation in~\eqref{direct-summand-witnessing-computation}
shows that $h_xjtl_x=tl_x$.

 For every index $x\in X$, the morphism $\id_Q-jh_x=R(\id_\Q-\fj\fh_x)$
is an automorphism of the object $Q\in\sA$, because the morphism
$\id_\Q-\fj\fh_x$ is an automorphism of the object $\Q\in\sB$ and
the functor $R$ (as any functor) takes isomorphisms to isomorphisms.
 Hence, similarly to~\eqref{locally-split-cover-computation},
the equations
$$
 (\id_Q-jh_x)jtl_x=jtl_x-jtl_x=0
$$
imply $jtl_x=0$.
 Since $j$~is a monomorphism, it follows that $tl_x=0$ for all
$x\in X$, and consequently $t=0$ and $L=0$, as desired.
\end{proof}

\begin{thm} \label{ab3-direct-limit-precover-cover-thm}
 Let\/ $\sB$ be an Ab3~category and\/ $\sC\subset\sB$ be a class of
objects closed under coproducts and direct summands.
 Let $(\fc_{y,x}\:\C_x\to\C_y)_{x<y\in X}$ be a diagram of objects\/
$\C_x\in\sC$, indexed by a directed poset $X$, and
let\/ $\D=\varinjlim_{x\in X}^\sB\C_x$ be its direct limit in
the category\/~$\sB$.
 Let\/ $\fp\:\coprod_{x\in X}^\sB\C_x\rarrow\D$ be the natural
epimorphism.
 Assume that
\begin{enumerate}
\item the morphism\/~$\fp$ is a\/ $\sC$\+precover in\/~$\sB$;
\item the object\/ $\D$ has a\/ $\sC$\+cover\/ $\fq\:\Q\rarrow\D$
in\/~$\sB$.
\end{enumerate}
 Put\/ $\L=\ker(\fq)$.
 Then, for every Ab5~category\/ $\sA$ and colimit-preserving functor
$R\:\sB\rarrow\sA$, one has $R(\L)=0$.
\end{thm}

\begin{proof}
 The proof is similar to that of
Theorem~\ref{ab5-direct-limit-precover-cover-thm}.
 By Proposition~\ref{direct-limit-epi-quasi-split},
the epimorphism~$\fp$ is quasi-split; so its kernel~$\ii$ is
a quasi-split monomorphism.
 By Lemma~\ref{quasi-split-is-locally-split}, it follows that
the morphism~$\ii$ is locally split.
 Hence Proposition~\ref{ab3-locally-split-precover-and-cover-prop} is
applicable.
\end{proof}

\Section{Covers of Direct Limits in Contramodule Categories}
\label{contramodule-direct-limits-secn}

 In this section we specialize the results of
Section~\ref{ab3-secn} to the case of the category $\sB=\R\contra$
of left contramodules over a topological ring~$\R$.

 A short exact sequence of left $\R$\+contramodules $0\rarrow\K
\rarrow\C\rarrow\D\rarrow0$ is said to be \emph{contratensor pure}
(or \emph{c\+pure} for brevity) if, for every discrete right
$\R$\+module $\N$, the short exact sequence of abelian groups
$0\rarrow \N\ocn_\R\K\rarrow\N\ocn_\R\C\rarrow\N\ocn_\R\D\rarrow0$
is exact, or equivalently, the map $\N\ocn_\R\K\rarrow\N\ocn_\R\C$
is injective.
 If a short exact sequence $0\rarrow\K\rarrow\C\rarrow\D\rarrow0$ in
$\R\contra$ is c\+pure, we will say that $\K\rarrow\C$ is
a \emph{c\+pure monomorphism} and $\C\rarrow\D$ is
a \emph{c\+pure epimorphism}.

 The functor of contratensor product $\N\ocn_\R{-}\:\R\contra\rarrow
\Ab$ takes values in the category of abelian groups $\sA=\Ab$ (which
has exact direct limits) and preserves all colimits (being a left
adjoint functor).
 So any f\+pure exact sequence (monomorphism, or epimorphism)
in $\sB=\R\contra$ is c\+pure.

 The next corollary is a generalization of~\cite[Lemma~2.2]{Pproperf}.

\begin{cor} \label{direct-limits-1-strictly-flat}
 The class of all\/ $1$\+strictly flat left\/ $\R$\+contramodules
is closed under direct limits in\/ $\R\contra$.
\end{cor}

\begin{proof}
 Let $(f_{y,x}\:\F_x\to\F_y)_{x<y\in X}$ be a diagram of $1$\+strictly
flat left $\R$\+contramodules $\F_x$, indexed by a directed poset~$X$.
 By Proposition~\ref{direct-limit-epi-quasi-split}
and Corollary~\ref{quasi-split-functor-pure}, the short exact sequence
of $\R$\+contramodules
$$
 0\lrarrow\K\overset\ii\lrarrow\coprod_{x\in X}^{\R\contra}\F_x
 \overset\fp\lrarrow\varinjlim_{x\in X}^{\R\contra}\F_x\lrarrow0
$$
is functor pure, hence (in particular)  contratensor pure.
 Since the $\R$\+contramodule $\coprod_{x\in X}^{\R\contra}\F_x$
is $1$\+strictly flat by~\cite[Lemma~2.1]{Pproperf}, it follows that
the $\R$\+contramodule $\varinjlim_{x\in X}^{\R\contra}\F_x$ is
$1$\+strictly flat (see the discussion in~\cite[Section~2]{Pproperf}).
\end{proof}

 Let $\sC\subset\R\contra$ be a class of left $\R$\+contramodules closed
under coproducts and direct summands.

\begin{cor} \label{contramodule-precover-cover-cor}
 Let $(\fc_{y,x}\:\C_x\to\C_y)_{x<y\in X}$ be a diagram of left\/
$\R$\+contramodules\/ $\C_x\in\sC\subset\R\contra$, indexed by
a directed poset $X$, and let\/ $\D=\varinjlim_{x\in X}^{\R\contra}\C_x$
be its direct limit in the category\/ $\R\contra$.
 Let\/ $\fp\:\coprod_{x\in X}^{\R\contra}\C_x\rarrow\D$ be the natural
epimorphism.
 Assume that
\begin{enumerate}
\item the morphism\/~$\fp$ is a\/ $\sC$\+precover in\/~$\R\contra$;
\item the\/ $\R$\+contramodule\/ $\D$ has a\/ $\sC$\+cover\/
$\fq\:\Q\rarrow\D$ in\/~$\R\contra$.
\end{enumerate}
 Put\/ $\L=\ker(\fq)$.
 Then, for every open right ideal\/ $\I\subset\R$, one has\/
$\I\tim\L=\L$.
\end{cor}

\begin{proof}
 In the context of Theorem~\ref{ab3-direct-limit-precover-cover-thm},
set $\sB=\R\contra$, \ $\sA=\Ab$, and $R=\N\ocn_\R{-}$, where $\N$ is
a discrete right $\R$\+module.
 Then we can conclude that $\N\ocn_\R\L=0$.
 In particular, for $\N=\R/\I$ and any left $\R$\+contramodule $\C$,
one has $\N\ocn_\R\C=\C/(\I\tim\C)$; so we get the desired equation
$\L=\I\tim\L$.
\end{proof}

\begin{cor}
 Suppose that the topological ring\/ $\R$ has a countable base of
neighborhoods of zero.
 Let $(\fc_{y,x}\:\C_x\to\C_y)_{x<y\in X}$ be a diagram of left\/
$\R$\+contramodules\/ $\C_x\in\sC\subset\R\contra$, indexed by
a directed poset $X$, and let\/ $\D=\varinjlim_{x\in X}^{\R\contra}\C_x$
be its direct limit in the category\/ $\R\contra$.
 Let\/ $\fp\:\coprod_{x\in X}^{\R\contra}\C_x\rarrow\D$ be the natural
epimorphism.
 Assume that
\begin{enumerate}
\item the morphism\/~$\fp$ is a\/ $\sC$\+precover in\/~$\R\contra$;
\item the\/ $\R$\+contramodule\/ $\D$ has a\/ $\sC$\+cover
in\/~$\R\contra$.
\end{enumerate}
 Then one has\/ $\D\in\sC$ and the epimorphism\/~$\fp$ is split.
\end{cor}

\begin{proof}
 By the contramodule Nakayama lemma~\cite[Lemma~6.14]{PR}, the equations
$\I\tim\L=\L$ for a fixed left $\R$\+contramodule $\L$ and all the open
right ideals $\I\subset\R$ imply $\L=0$.
 So the assertion follows from
Corollary~\ref{contramodule-precover-cover-cor}.

 Note that this version of contramodule Nakayama lemma does \emph{not}
hold without the assumption of a countable base of neighborhoods of
zero in $\R$, generally speaking.
 For a counterexample, see~\cite[Remark~6.3]{Pcoun}.
\end{proof}

 A left $\R$\+contramodule $\C$ is said to be \emph{separated} if
the intersection of its subgroups $\I\tim\C$, taken over all the open
right ideals $\I\subset\R$, vanishes.
 (See~\cite[Section~7.3]{PS1} or~\cite[Section~5]{Pcoun}
for the discussion.)

\begin{cor} \label{separated-precover-cover-cor}
 Suppose that the class\/ $\sC\subset\R\contra$ consists of separated
left\/ $\R$\+contramodules.
 Let $(\fc_{y,x}\:\C_x\to\C_y)_{x<y\in X}$ be a diagram of left\/
$\R$\+contramodules\/ $\C_x\in\sC\subset\R\contra$, indexed by
a directed poset $X$, and let\/ $\D=\varinjlim_{x\in X}^{\R\contra}\C_x$
be its direct limit in the category\/ $\R\contra$.
 Let\/ $\fp\:\coprod_{x\in X}^{\R\contra}\C_x\rarrow\D$ be the natural
epimorphism.
 Assume that
\begin{enumerate}
\item the morphism\/~$\fp$ is a\/ $\sC$\+precover in\/~$\R\contra$;
\item the\/ $\R$\+contramodule\/ $\D$ has a\/ $\sC$\+cover
in\/~$\R\contra$.
\end{enumerate}
 Then one has\/ $\D\in\sC$ and the epimorphism\/~$\fp$ is split.
\end{cor}

\begin{proof}
 In the context of Corollary~\ref{contramodule-precover-cover-cor},
we have $\L\subset\Q$ and $\Q\in\sC$.
 Hence
$$
 \bigcap\nolimits_{\I\subset\R}(\I\tim\L)\,\subset\,
 \bigcap\nolimits_{\I\subset\R}(\I\tim\Q)\,=\,0,
$$
where the intersection is taken over all the open right ideals
$\I\subset\R$.
 Thus the equations $\I\tim\L=\L$ for all open right ideals
$\I$ imply $\L=0$.
\end{proof}

 The next corollary is the main result of this section, and
the promised generalization of Corollary~\ref{bass-flat-contra-cover}.

\begin{cor} \label{direct-limits-proj-contramodules-no-proj-covers}
 Let $(\ff_{y,x}\:\P_x\to\P_y)_{x<y\in X}$ be a diagram of projective
left\/ $\R$\+contra\-modules\/ $\P_x\in\R\contra_\proj$, indexed by
a directed poset $X$, and let\/ $\F=\varinjlim_{x\in X}^{\R\contra}\P_x$
be its direct limit in the category\/ $\R\contra$.
 Assume that the\/ $\R$\+contramodule\/ $\F$ has a projective cover
in\/ $\R\contra$.
 Then\/ $\F$ is a projective left\/ $\R$\+contramodule.
\end{cor}

\begin{proof}
 Let $\sC=\R\contra_\proj$ be the class of all projective left
$\R$\+contramodules.
 All projective $\R$\+contramodules are separated, so
Corollary~\ref{separated-precover-cover-cor} is applicable.

 Furthermore, the left $\R$\+contramodule $\coprod_{x\in X}\P_x$
is projective and the morphism $\fp\:\coprod_{x\in X}\P_x\rarrow\F$
is an epimorphism; so $\fp$~is a projective precover.
 Hence the condition~(1) is satisfied.
 The condition~(2) is satisfied by assumption.
 We can conclude that $\F\in\sC=\R\contra_\proj$.
\end{proof}

\Section{Applications to the Enochs Conjecture}

 For the benefit of a reader not necessarily familiar with the context,
let us recall the statement of the conjecture~\cite[Section~5.4]{GT}
(cf.~\cite[Section~5]{AST}).

\begin{conj}[a question of Enochs]
 Let $A$ be an associative ring and\/ $\sL\subset A\modl$ be a class of
left $A$\+modules.
 Assume that every left $A$\+module has an\/ $\sL$\+cover.
 Then the class of modules\/ $\sL$ is closed under direct limits in
$A\modl$.
\end{conj}

 Let $\sA$ be a cocomplete abelian category and $M\in\sA$ be an object.
 Then we denote by $\Add(M)\subset A\modl$ the class of all direct
summands of coproducts of copies of $M$ in~$\sA$.
 In this section we mostly discuss certain results in the direction of
 the Enochs conjecture \emph{for the class of objects\/
$\sL=\Add(M)$}.

\subsection{Cotorsion pairs with the right class closed under
direct limits}
 Given a class of objects $\sL$ in a cocomplete abelian category $\sA$,
we will denote by $\varinjlim^\sA\sL\subset\sA$ the class of all direct
limits of objects from~$\sL$.
 This means the direct limits in $\sA$ of diagrams of objects of $\sL$
indexed by directed posets.

 For a pair of objects $M$, $N$ in a cocomplete abelian category $\sA$,
let us denote by $\PExt^1_\sA(M,N)\subset\Ext^1_\sA(M,N)$
the abelian group of all equivalence classes of f\+pure short exact
sequences $0\rarrow N\rarrow{?}\rarrow M\rarrow0$ in~$\sA$.
 In other words, $\PExt^1_\sA({-},{-})$ is the $\Ext^1$ group in
the functor pure exact structure on the category $\sA$
(see~\cite[Section~8]{BP3}).

\begin{cor} \label{pext1-corollary}
 Let\/ $\sA$ be an Ab5~category and $M\in\sA$ be an object.
 Suppose that $\PExt^1_\sA(M,E)=0$ for all objects
$E\in\varinjlim^\sA\Add(M)$.
 Let $D\in\varinjlim^\sA\Add(M)$ be an object having 
an\/ $\Add(M)$\+cover in\/~$\sA$.
 Then $D\in\Add(M)$.
\end{cor}

\begin{proof}
 This is a corollary of
Theorem~\ref{ab5-direct-limit-precover-cover-thm}.
 Let $(c_{y,x}\:C_x\to C_y)_{x<y\in X}$ be a diagram of objects
$C_x\in\Add(M)$ such that $D=\varinjlim^\sA_{x\in X}C_x$.
 Then the short exact sequence
\begin{equation} \label{Ab5-generalized-telescope}
 0\lrarrow K\overset i\lrarrow\coprod_{x\in X}C_x\overset p
 \lrarrow D\lrarrow0
\end{equation}
(cf.~\eqref{generalized-telescope}) is f\+pure exact in~$\sA$
by Proposition~\ref{direct-limit-epi-quasi-split}
and Corollary~\ref{quasi-split-functor-pure}
(see also~\cite[Example~8.4]{BP3}).
 Furthermore, following the proof of
Proposition~\ref{direct-limit-epi-quasi-split} and keeping in mind
that the direct limits in $\sA$ are exact by assumption, the object
$K$ in the short exact sequence~\eqref{Ab5-generalized-telescope} is
the direct limit of the objects $K_z$, which belong to $\Add(M)$.
 By assumption, it follows that $\PExt^1_\sA(M,K)=0$.
 Since the sequence~\eqref{Ab5-generalized-telescope} is f\+pure exact,
we can conclude that any morphism $M\rarrow D$ can be lifted to
a morphism $M\rarrow\coprod_{x\in X}C_x$.
 In other words, this means that the morphism
$p\:\coprod_{x\in X}C_x\rarrow D$ is an $\Add(M)$\+precover.
 According to Theorem~\ref{ab5-direct-limit-precover-cover-thm},
we have $D\in\Add(M)$.
\end{proof}

 Let $\sA$ be an abelian category.
 A pair of full subcategories $(\sL,\sE)$ in $\sA$ is said to be
a \emph{cotorsion pair} if
\begin{itemize}
\item $\sE$ is the class of all objects $A\in\sA$ such that
$\Ext^1_\sA(L,A)=0$ for all $L\in\sL$; and
\item $\sL$ is the class of all objects $A\in\sA$ such that
$\Ext^1_\sA(A,E)=0$ for all $E\in\sE$.
\end{itemize}
 The intersection $\sL\cap\sE\subset\sA$ is called the \emph{kernel}
of the cotorsion pair $(\sL,\sE)$.

\begin{appl} \label{cotorsion-pair-appl}
 Let\/ $\sA$ be an Ab5\+category and $(\sL,\sE)$ be a cotorsion pair
in~$\sA$.
 Assume that the class of objects\/ $\sE\subset\sA$ is closed under
direct limits, and let $M\in\sL\cap\sE$ be an object of the kernel.
 Let $D\in\varinjlim^\sA\Add(M)$ be an object having
an\/ $\Add(M)$\+cover in\/~$\sA$.
 Then $D\in\Add(M)$.
\end{appl}

\begin{proof}
 In any cotorsion pair $(\sL,\sE)$ in a cocomplete abelian category
$\sA$, the left class $\sL$ is closed under coproducts, and both
the classes $\sL$ and $\sE$ are closed under direct summands.
 In the situation at hand, the class $\sE\subset\sA$ is closed under
coproducts by assumption.
 Hence for an object $M\in\sL\cap\sE$ we have
$\Add(M)\subset\sL\cap\sE$.
 Since, moreover, we have assumed that the class $\sE\subset\sA$ is
closed under direct limits, we have $\varinjlim^\sA\Add(M)\subset\sE$.
 It follows that $\Ext^1_\sA(M,E)=0$ for all
$E\in\varinjlim^\sA\Add(M)$, as $(\sL,\sE)$ is a cotorsion pair
and $M\in\sL$.

 Applying Corollary~\ref{pext1-corollary}, we conclude that
$D\in\Add(M)$.
\end{proof}

 The Enochs conjecture for the left class $\sL$ of an \emph{$n$\+tilting
cotorsion pair} in an Ab5\+category $\sA$ can be deduced from
Application~\ref{cotorsion-pair-appl}.
 Let $T\in\sA$ be an \emph{$n$\+tilting object} in the sense of
the paper~\cite{PS1} (see~\cite[Section~11]{BP3} for a brief summary),
and let $(\sL,\sE)$ be the corresponding tilting cotorsion pair in~$\sA$.

 By~\cite[Proposition~12.3]{BP3}, the tilting class $\sE$ is closed
under direct limits in~$\sA$.
 Furthermore, by~\cite[Theorem~3.4]{PS1}, the tilting cotorsion pair
$(\sL,\sE)$ is \emph{complete} and \emph{hereditary}, its kernel
$\sL\cap\sE$ coincides with the class $\Add(T)\subset\sA$
(by~\cite[Lemma~3.2(b)]{PS1}), and the \emph{coresolution dimension} of
objects of $\sA$ with respect to the coresolving subcategory $\sE$
does not exceed~$n$ (by~\cite[Lemma~3.1]{PS1}).
 We refer to~\cite[Section~3]{PS1} for the definitions of the terms
involved.
 These are the properties of the cotorsion pair $(\sL,\sE)$ that we will
use.

\begin{cor} \label{tilting-corollary}
 Let\/ $\sA$ be an Ab5\+category and $T\in\sA$ be an $n$\+tilting object
(where $n\ge0$ is an integer).
 Let $(\sL,\sE)$ be the $n$\+tilting cotorsion pair in\/ $\sA$
induced by~$T$.
 Assume that all the objects of\/ $\sA$ have\/ $\sL$\+covers.
 Then the class of objects\/ $\sL\subset\sA$ is closed under direct
limits.
\end{cor}

\begin{proof}
 It suffices to assume that every object of the class
$\varinjlim^\sA\Add(T)\subset\sA$ has an $\sL$\+cover.
 Then, by~\cite[Lemma~10.2]{BP3}, every object of
$\varinjlim^\sA\Add(T)$ as an $\Add(T)$\+cover (since
$\varinjlim^\sA\Add(T)\subset\varinjlim^\sA\sE\subset\sE$
by~\cite[Proposition~12.3]{BP3} and $\sL\cap\sE=\Add(T)$
by~\cite[Lemma~3.2(b)]{PS1}). 
  By Application~\ref{cotorsion-pair-appl}, it follows that the class
$\Add(T)$ is closed under direct limits in~$\sA$.
 According to~\cite[Corollary~12.4\,(ii)\,$\Rightarrow$\,(i)]{BP3},
this is equivalent to the class $\sL\subset\sA$ being closed under
direct limits.
 We refer to~\cite[Theorem~13.2\,(1)\,$\Rightarrow$\,\allowbreak(2)%
\,$\Rightarrow$\,\allowbreak(5)\,$\Rightarrow$\,\allowbreak(6)%
\,$\Rightarrow$\,(3)]{BP3} for a more detailed discussion.
 This argument only uses the basic properties of $n$\+tilting cotorsion
pairs in abelian categories.
\end{proof}

 So we have obtained a rather simple elementary proof of some of
the results in (the tilting case of) \cite[Corollary~5.5]{AST},
extended from the module categories to abelian categories with
exact direct limits.

\subsection{Weakly countably generated modules}
 Let $A$ be an associative ring and $M\in A\modl$ be a left $A$\+module.

 A result which we call the ``generalized tilting theory'' establishes
an equivalence $\Add(M)\simeq\R\contra_\proj$ between the full
subcategory $\Add(M)\subset A\modl$ and the full subcategory of
projective contramodules $\R\contra_\proj\subset\R\contra$ over
a certain topological ring~$\R$ \cite[Theorems~7.1 and~9.9]{PS1},
\cite[Section~2]{BP3}.
 Moreover, this equivalence of categories is obtained as a restriction
of a pair of adjoint functors
\begin{equation} \label{psi-phi-adjunction}
 \Psi\:A\modl\,\rightleftarrows\,\R\contra\,:\!\Phi
\end{equation}
where the right adjoint functor $\Psi$ can be computed as
$\Psi(N)=\Hom_A(M,N)$, while the left adjoint functor $\Phi$ is
the contratensor product $\Phi(\C)=M\ocn_\R\C$.
 In particular, the left $A$\+module $M\in\Add(M)$ corresponds to
the free left $\R$\+contramodule with one generator $\R=\R[[*]]\in
\R\contra_\proj$ \cite[Proposition~7.3]{PS1}, \cite[Theorem~3.13]{PS3}.

 Here the underlying abstract ring of the topological ring $\R$ is
the ring of endomorphisms $\R=\Hom_A(M,M)$ of the $A$\+module~$M$.
 We use the convention in which the ring $\R$ acts on the module $M$
on the right, making $M$ an $A$\+$\R$\+bimodule.
 The composition multiplication in $\Hom_A(M,M)$ is defined
accordingly.
 However, there is a certain flexibility in the choice of a topology
on the ring~$\R$.
 In particular, one can use the \emph{finite
topology}~\cite[Theorem~7.1]{PS1}, \cite[Example~3.7(1)]{PS3} or
the \emph{weakly finite topology}~\cite[Theorem~9.9]{PS1},
\cite[Example~3.10(2)]{PS3}.

 Let us briefly recall the definitions of these topologies.
 A left $A$\+module $E$ is said to be \emph{weakly finitely generated}
(or \emph{dually slender}~\cite{Zem}) if the natural map
$\bigoplus_{x\in X}\Hom_A(E,N_x)\rarrow\Hom_A(E,\bigoplus_{x\in X}N_x)$
is an isomorphism for every family of left $A$\+modules
$(N_x)_{x\in X}$.
 In the \emph{finite topology} on the ring $\Hom_A(M,M)$,
the annihilators of finitely generated submodules $F\subset M$ form
a base of neighborhoods of zero.
 In the \emph{weakly finite topology} on the ring $\Hom_A(M,M)$,
the annihilators of weakly finitely generated submodules $E\subset M$
form a base of neighborhoods of zero.
 In any one of these two topologies, $\R=\Hom_A(M,M)$ is
a complete, separated right linear topological ring.

 Let us say that a left $A$\+module $M$ is \emph{weakly countably
generated} if there exists a suitable complete, separated right linear
topological ring structure on the ring $\R=\Hom_A(M,M)$ with
a countable base of neighborhoods of zero in~$\R$.
 Here a ``suitable topological ring structure on the endomorphism
ring~$\R$'' means that there is a pair of adjoint functors $\Psi$ and
$\Phi$ \eqref{psi-phi-adjunction} whose restrictions to the full
subcategories $\Add(M)\subset A\modl$ and $\R\contra_\proj\subset
\R\contra$ are mutually inverse equivalences of categories
$\Add(M)\simeq\R\contra_\proj$ assigning the free left
$\R$\+contramodule with one generator $\R=\R[[*]]$ to
the left $A$\+module~$M$.

\begin{lem}
 Let $M$ be a left $A$\+module such that $M=\sum_{i=1}^\infty E_i$,
where $(E_i\subset M)_{i=1}^\infty$ is a countable set of submodules
in $M$ and all the $A$\+modules $E_i$ are weakly finitely generated.
 Then the $A$\+module $M$ is weakly countably generated.
\end{lem}

\begin{proof}
 We consider the weakly finite topology on the ring~$\R$.
 In view of~\cite[Theorem~9.9]{PS1} and~\cite[Theorem~3.13]{PS3},
it only remains to show that $\R$ has a countable base of
neighborhoods of zero.
 We claim that the annihilators of the submodules $E_1+\dotsb+E_n
\subset M$, \ $n\ge1$, form such a base.

 Indeed, a finite sum of weakly finitely generated submodules is
clearly weakly finitely generated.
 So it remains to check that, for any weakly finitely generated
submodule $E\subset M$ there exists $n\ge1$ such that
$E\subset\sum_{i=1}^nE_i$.
 For this purpose, put $N_m=M/\sum_{i=1}^mE_i$ and consider
the family of left $A$\+modules $(N_m)_{m=1}^\infty$.
 Consider the left $A$\+module morphism $f\:M\rarrow\prod_{m=1}^\infty
N_m$ whose components are the epimorphisms $M\rarrow N_m$.
 Since $M=\sum_{i=1}^\infty E_i$, the image of the map~$f$
is contained in the submodule $\bigoplus_{m=1}^\infty N_m
\subset\prod_{m=1}^\infty N_m$.

 So we have an $A$\+module morphism $g\:M\rarrow
\bigoplus_{m=1}^\infty N_m$.
 Denote by $h=g|_E\:E\rarrow\bigoplus_{m=1}^\infty N_m$
the restriction of the morphism~$g$ to the submodule $E\subset M$.
 Since the $A$\+module $E$ is weakly finitely generated, there
exists an integer $n\ge1$ such that the image of the morphism~$h$
is contained in the submodule $\bigoplus_{m=1}^{n-1}N_m
\subset\bigoplus_{m=1}^\infty N_m$.
 This means exactly that $E\subset\sum_{i=1}^n E_i$.
\end{proof}

 Let $A$ be an associative ring and $M$ be a left $A$\+module.
 Let $M\overset{f_1}\rarrow M\overset{f_2}\rarrow M\overset{f_3}
\rarrow\dotsb$ be a countable direct system of copies of~$M$.
 Consider the related telescope short exact sequence of left
$A$\+modules
\begin{equation} \label{telescope}
 0\lrarrow\bigoplus\nolimits_{n=0}^\infty M\overset i\lrarrow
 \bigoplus\nolimits_{n=0}^\infty M\overset p\lrarrow
 \varinjlim\nolimits_{n\ge0}M\lrarrow0.
\end{equation}
 Following~\cite[Section~6]{BP3}, we will say that a left $A$\+module
$M$ satisfies the \emph{telescope Hom exactness condition}
(\emph{THEC}) if, for every sequence of endomorphisms~$f_1$, $f_2$,
$f_3$,~\dots~$\in\Hom_A(M,M)$ the telescope short exact
sequence~\eqref{telescope} stays exact after applying the functor
$\Hom_A(M,{-})$.
 It is worth noticing that the short exact sequence~\eqref{telescope}
staying exact after $\Hom_A(M,{-})$ is applied means precisely that
the epimorphism $\bigoplus_{n=0}^\infty M\overset p\rarrow D$ is
an \emph{$\Add(M)$\+precover} of the $A$\+module
$D=\varinjlim_{n\ge0}M$.

 Furthermore, let us say that an $A$\+module $N$ is
\emph{$\Sigma$\+pure-rigid} if the pure Ext group
$\PExt_A^1(N,N^{(\omega)})$ vanishes.
 All $\Sigma$\+pure-rigid modules satisfy THEC (since the short
exact sequence~\eqref{telescope} is pure and the pullback of
a pure exact sequence~\eqref{telescope} with respect to any
$A$\+module morphism $M\rarrow D$ is pure).
 There are also other sufficient conditions for THEC discussed
in~\cite{BP3}.

 A family of left $A$\+modules $(M_\theta)_{\theta\in\Theta}$ is
said to be \emph{locally T\+nilpotent} if for every sequence of
indices $\theta_1$, $\theta_2$, $\theta_3$,~\dots~$\in\Theta$,
every sequence of nonisomorphisms $f_i\in
\Hom_A(M_{\theta_i},M_{\theta_{i+1}})$, \,$i=1$, $2$, $3$,~\dots,
and every element $b\in M_{\theta_1}$, there exists an integer $n\ge1$
such that $f_nf_{n-1}\dotsm f_1(b)=0$ in $M_{\theta_{n+1}}$.
 A left $A$\+module $M$ is said to have a \emph{perfect
decomposition}~\cite{AS} if there exists a locally T\+nilpotent family
of left $A$\+modules $(M_\theta)_{\theta\in\Theta}$ such that
$M\simeq\bigoplus_{\theta\in\Theta}M_\theta$.

\begin{appl} \label{weakly-countably-generated-application}
 Let $A$ be an associative ring and $M$ be a weakly countably generated
left $A$\+module.
 Suppose that $M$ satisfies THEC.
 Assume that all the countable direct limits of copies of $M$ in
the category of left $A$\+modules $A\modl$ have\/ $\Add(M)$\+covers.
 Then the class of objects\/ $\Add(M)\subset A\modl$ is closed under
direct limits, and the $A$\+module $M$ has a perfect decomposition.
\end{appl}

\begin{proof}
 By assumption, we have a natural equivalence $\Add(M)\simeq
\R\contra_\proj$.
 Under this equivalence, direct systems $M\overset{f_1}\rarrow M
\overset{f_2}\rarrow M\overset{f_3}\rarrow\dotsb$ of copies of
the $A$\+module $M$ corresponds to direct systems
$\R\overset{*a_1}\rarrow\R\overset{*a_2}\rarrow\R\overset{*a_3}
\rarrow\dotsb$ of copies of the left $\R$\+contramodule $\R$
(where $a_n=f_n\in\Hom_A(M,M)=\R$).
 The direct limit of the latter direct system is a Bass flat
left $\R$\+contramodule~$\B$.

 Following the proof
of~\cite[Corollary~6.7\,(1)\,$\Rightarrow$\,(2)]{BP3}, under
our THEC assumption the left $A$\+module $D=\varinjlim_{n\ge1}M$
has an $\Add(M)$\+cover if and only if the left $\R$\+contramodule
$\B$ has a projective cover.
 If this is the case, then by Corollary~\ref{bass-flat-contra-cover}
the left $\R$\+contramodule $\B$ is projective.

 Alternatively, one can use
Theorem~\ref{ab5-direct-limit-precover-cover-thm}
(in the category of left $A$\+modules $\sA=A\modl$) in order
to conclude, from the assumptions that $p$~is an $\Add(M)$\+precover
and $D$ has an $\Add(M)$\+cover, that the epimorphism of
left $A$\+modules $p\:\bigoplus_{n=1}^\infty M\rarrow D$ splits.
 Consequently, we have $D\in\Add(M)$.
 By~\cite[Corollary~6.7\,(3)\,$\Rightarrow$\,(4)]{BP3}, the left
$\R$\+contramodule $\B$ is projective.

 As this holds for all Bass flat left $\R$\+contramodules,
by~\cite[Proposition~4.3 and Lemma~6.3]{Pproperf} it follows that
all discrete right $\R$\+modules are coperfect (i.~e., all descending
chains of cyclic submodules in discrete right $\R$\+modules terminate).
 Now \cite[Theorem~12.4]{PS3} is applicable, since the topological ring
$\R$ has a countable base of neighborhoods of zero by assumption.
 Thus we can conclude that the topological ring $\R$ is topologically
left perfect (that is, the Jacobson radical $\H\subset\R$ is
topologically left T\+nilpotent and strongly closed in $\R$, and
the quotient ring $\R/\H$ in its quotient topology is
topologically semisimple).

 Hence, according
to~\cite[Theorem~14.1\,(iv)\,$\Rightarrow$\,(iii$'$)]{PS3}, the class of
all projective left $\R$\+contra\-modules is closed under direct limits
in $\R\contra$.
 Following~\cite[Corollary~9.9]{PS3}, this means that the full
subcategory $\Add(M)\subset A\modl$ has split direct limits.
 In particular, by~\cite[Lemma~9.2(b)]{PS3} or~\cite[Lemma~6.5]{BP3},
\,$\Add(M)$ is closed under direct limits in $A\modl$.
 By~\cite[Theorem~1.4]{AS}, split direct limits in $\Add(M)$ imply
that the $A$\+module $M$ has perfect decomposition.
\end{proof}

\medskip
\noindent\textbf{Acknowledgement.}
 We are grateful to Jan \v Saroch for very helpful discussions and
communications.
 We also wish to thank Jan Trlifaj for illuminating conversations.
 We are grateful to an anonymous referee for careful reading of
the manuscript and several useful suggestions.
 The first-named author was partially supported by MIUR-PRIN
(Categories, Algebras: Ring-Theoretical and Homological
Approaches-CARTHA) and DOR1828909 of Padova University.
 The second-named author is supported by research plan
RVO:~67985840. 
 The third-named author was supported by the Czech Science
Foundation grant number 17-23112S.

\end{document}